\documentclass[11pt]{article}

\usepackage{amssymb, amsfonts, amsmath, amsthm, mathabx}

\newtheorem{theorem}{Theorem}[section]
\newtheorem{lemma}[theorem]{Lemma}
\newtheorem{proposition}[theorem]{Proposition}
\newtheorem{corollary}[theorem]{Corollary}

\theoremstyle{definition}
\newtheorem{definition}[theorem]{Definition}
\newtheorem{notation}[theorem]{Notation}
\newtheorem{remark}[theorem]{Remark}

\newcommand{\cA}{ {\cal A} }
\newcommand{\bB}{ {\mathbb B} }
\newcommand{\cB}{ {\cal B} }
\newcommand{\bC}{ {\mathbb C} }

\newcommand{\cD}{ {\cal D} }
\newcommand{\cF}{ {\cal F} }
\newcommand{\fF}{ {\mathfrak F} }
\newcommand{\cH}{ {\cal H} }

\newcommand{\cK}{ {\cal K} }
\newcommand{\cL}{ {\cal L} }
\newcommand{\cM}{ {\cal M} }
\newcommand{\fM}{ {\mathfrak M} }
\newcommand{\bN}{ {\mathbb N} }

\newcommand{\bR}{ {\mathbb R} }
\newcommand{\cR}{{\cal R}}

\newcommand{\fV}{ {\mathfrak V} }

\newcommand{\bZ}{ {\mathbb Z} }

\newcommand{\dalg}{ \cD_{\mathrm{alg}} }
\newcommand{\ncdirac}{ \delta }
\newcommand{\ncserk}{ \bC_0 \langle \langle z_1, \ldots ,
                           z_k \rangle \rangle }
\newcommand{\cf}{ \mbox{Cf} }

\newcommand{\geqgeq}{ \gg }
\newcommand{\Int}{ \mbox{Int} }
\newcommand{\leqleq}{ \ll }
\newcommand{\oo}{ \mbox{o} }
\newcommand{\opi}{ {\mbox{o}}_{\pi} }
\newcommand{\pvac}{ P_{\Omega} }

\begin{document}

$\ $

\begin{center}
{\bf\Large Multi-variable subordination distributions for}

\vspace{6pt}

{\bf\Large free additive convolution}

\vspace{20pt}

{\large Alexandru Nica
\footnote{Research supported by a Discovery Grant from 
          NSERC, Canada.}    }

\vspace{10pt}

\end{center}

\begin{abstract}
Let $k$ be a positive integer and let $\cD_c (k)$ denote the 
space of joint distributions for $k$-tuples of selfadjoint elements 
in $C^*$-probability space. The paper studies the concept of 
``subordination distribution of $\mu \boxplus \nu$ with respect to 
$\nu$'' for $\mu , \nu \in \cD_c (k)$, where $\boxplus$ is the 
operation of free additive convolution on $\cD_c (k)$. The main 
tools used in this study are combinatorial properties of 
$R$-transforms for joint distributions and a related operator 
model, with operators acting on the full Fock space. 

Multi-variable subordination turns out to have nice relations to a 
process of evolution towards $\boxplus$-infinite divisibility on 
$\cD_c (k)$ that was recently found by Belinschi and Nica
(arXiv:0711.3787). Most notably, one gets better insight into a 
connection which this process was known to have with free Brownian 
motion.
\end{abstract}

\vspace{6pt}

\begin{center}
{\bf\large 1. Introduction and statements of results}
\end{center}
\setcounter{section}{1}

\noindent
The free additive convolution $\boxplus$ is a binary operation 
on the space of probability distributions on $\bR$, reflecting the 
addition operation for free random variables in a non-commutative
probability space. A significant fact in its theory (see \cite{V93}, 
\cite{B98}, \cite{V00}) is that the Cauchy transform of the 
distribution $\mu \boxplus \nu$ is subordinated to the Cauchy 
transforms of $\mu$ and of $\nu$, as analytic functions on the 
upper half-plane $\bC^{+}$. Thus (choosing for instance to discuss
subordination with respect to $\nu$) one has an analytic 
subordination function $\omega : \bC^{+} \to \bC^{+}$ such that 
\[
G_{\mu \boxplus \nu} (z) = G_{\nu} ( \, \omega (z) \, ), \ \ 
\forall \, z \in \bC^{+},
\]
where $G_{\mu \boxplus \nu}$ and $G_{\nu}$ are the Cauchy 
transforms of $\mu \boxplus \nu$ and of $\nu$, respectively.
Moreover, the subordination function $\omega$ can be 
identified as the reciprocal Cauchy transform of a uniquely 
determined probability distribution $\sigma$ on $\bR$. Following
\cite{L07}, this $\sigma$ will be denoted as ``$\mu \boxright \nu$''.
The name used in \cite{L07} for $\sigma = \mu \boxright \nu$ is 
``$s$-free additive convolution of $\mu$ and $\nu$'', in relation
to a suitably tailored concept of ``$s$-freeness'' that is also
introduced in \cite{L07}. Since $s$-freeness is only
marginally addressed in the present paper, $\mu \boxright \nu$ will 
just be called here the {\em subordination distribution} 
of $\mu \boxplus \nu$ with respect to $\nu$.

The goal of the present paper is to introduce and study the 
analogue for $\mu \boxright \nu$ in a multi-variable framework 
where $\mu, \nu$ become joint distributions of $k$-tuples of 
selfadjoint elements in a $C^*$-probability space. The 
particular case $k=1$ corresponds of course to the framework 
of probability distributions on $\bR$ as discussed above, with 
$\mu, \nu$ compactly supported. The main tool used in the paper 
is the $R$-transform for joint distributions. In particular, the 
$k$-variable version of $\mu \boxright \nu$ is introduced in 
Definition \ref{def:1.1} below via an extension of the formula 
which describes $R_{\mu \boxright \nu}$ in terms of $R_{\mu}$ 
and $R_{\nu}$ in the case $k=1$. (The 1-variable motivation 
behind Definition \ref{def:1.1} is presented in Section 2A of 
the paper.)

It is convenient to write the definition for the $k$-variable 
version of $\mu \boxright \nu$ by allowing $\mu$ and $\nu$ 
to be any linear functionals on 
$\bC \langle X_1, \ldots , X_k \rangle$ (the algebra of 
polynomials in non-commuting indeterminates $X_1, \ldots , X_k$)
such that $\mu (1) = \nu (1) = 1$. The set of all such ``purely
algebraic'' distributions will be denoted by $\dalg (k)$. The 
main interest of the paper is in the smaller set of
``non-commutative $C^*$-distributions with compact support''
\[
\cD_c (k) := \left\{
\mu \in \dalg (k) \
\begin{array}{cl}
\vline & \mbox{$\mu$ can appear as joint 
               distribution for a $k$-tuple}            \\
\vline & \mbox{of selfadjoint elements in a 
               $C^*$-probability space}
\end{array}  \right\} .
\]
But in order to define $\boxright$ on $\cD_c (k)$ it comes 
in handy to first define it as a binary operation on $\dalg (k)$, 
and then prove that $\mu \boxright \nu \in \cD_c (k)$ whenever 
$\mu , \nu \in \cD_c (k)$. 

In the next definition and throughout the paper, $k$ is a 
positive integer denoting ``the number of variables'' that one 
is working with.

\vspace{6pt}

\begin{definition}  \label{def:1.1}
Let $\mu, \nu$ be distributions in $\dalg (k)$. The 
{\em subordination distribution} of $\mu \boxplus \nu$ with respect 
to $\nu$ is the distribution $\mu \boxright \nu \in \dalg (k)$ 
uniquely determined by the requirement that its $R$-transform is
\begin{equation}  \label{eqn:1.11}
R_{\mu \boxright \nu}  = 
R_{\mu} \Bigl( \, z_1 (1+M_{\nu}), \ldots , z_k (1+M_{\nu}) 
\, \Bigr) \cdot \bigl( 1+M_{\nu} \bigr)^{-1}.
\end{equation}
In (\ref{eqn:1.11}) $M_{\nu}$ is the moment series of $\nu$ and 
$( 1+M_{\nu} )^{-1}$ is the inverse of $1 + M_{\nu}$ with respect 
to multiplication, in the algebra 
$\bC \langle \langle z_1, \ldots , z_k \rangle \rangle$ of 
power series in the noncommuting indeterminates $z_1, \ldots , z_k$.
(A more detailed review of the notations used here can be found in 
Section 2C of the paper.)
\end{definition}

\begin{remark}  \label{rem:1.2}
$1^o$ From Equation (\ref{eqn:1.11}) it is clear that the 
$R$-transform of $\mu \boxright \nu$ depends linearly on the one 
of $\mu$. Since the $R$-transform linearizes $\boxplus$, this 
amounts to a form of ``$\boxplus$-linearity'' in the way how 
$\mu \boxright \nu$ depends on $\mu$. More precisely one has
\begin{equation}      \label{eqn:1.21}
( \mu_1 \boxplus \mu_2 ) \boxright \nu =
( \mu_1 \boxright \nu ) \boxplus ( \mu_2 \boxright \nu ),
\ \ \forall \, \mu_1, \mu_2, \nu \in \dalg (k),
\end{equation}
or, when looking at $\boxplus$-convolution powers,
\begin{equation}      \label{eqn:1.22}
( \mu^{\boxplus t} ) \boxright \nu =
( \mu \boxright \nu )^{\boxplus t}, 
\ \ \forall \, \mu, \nu \in \dalg (k), \ \forall \, t>0.
\end{equation}

\vspace{6pt}

$2^o$ The series $R_{\mu} \Bigl( \, z_1 (1+M_{\nu}), \ldots , 
z_k (1+M_{\nu}) \, \Bigr)$ appearing in (\ref{eqn:1.11}) bears a 
resemblance to a well-known ``functional equation for the 
$R$-transform'' (see Lecture 16 of \cite{NS06}), which says that 
\[
R_{\mu} \Bigl( \, z_1 (1+M_{\mu}), \ldots , 
z_k (1+M_{\mu}) \, \Bigr) = M_{\mu}, \ \ \forall \, 
\mu \in \dalg (k).
\]
One can actually invoke this functional equation in the 
particular case of Definition \ref{def:1.1} when $\nu = \mu$, to
obtain that 
\begin{equation}     \label{eqn:1.225}
R_{\mu \boxright \mu} = M_{\mu} \cdot (1+M_{\mu})^{-1}, \ \ 
\mu \in \dalg (k).
\end{equation}

The series $M_{\mu} \cdot (1+ M_{\mu})^{-1}$ is called the 
{\em $\eta$-series} of $\mu$, and plays an important role in the
study of connections between free and Boolean probability. In
particular, the relation between $R$-transforms and $\eta$-series
yields a special bijection $\bB : \dalg (k) \to \dalg (k)$, defined 
as follows: for every $\mu \in \dalg (k)$, $\bB ( \mu )$ is the 
unique distribution in $\dalg (k)$ which has 
\begin{equation}     \label{eqn:1.226}
R_{ \bB ( \mu ) } = \eta_{\mu}.
\end{equation}
$\bB$ is called the {\em Boolean Bercovici-Pata bijection} (first put 
into evidence in the 1-variable case in \cite{BP99}, then extended to
multi-variable framework in \cite{BN08}). This bijection has the 
important property that it carries $\cD_c (k)$ into itself and that 
$\bB ( \, \cD_c (k) \, )$ is precisely the set of 
distributions in $\cD_c (k)$ which are infinitely divisible with 
respect to $\boxplus$ (cf. Theorem 1 in \cite{BN08}).

By comparing (\ref{eqn:1.225}) to (\ref{eqn:1.226}), one draws 
the conclusion that 
\begin{equation}   \label{eqn:1.23}  
\mu \boxright \mu  =  \bB ( \mu ), \ \ \forall \, 
\mu \in \dalg (k).
\end{equation}
Equation (\ref{eqn:1.23}) can be generalized to a nice formula 
describing $\mu_1 \boxright \mu_2$ in the case when both $\mu_1$ 
and $\mu_2$ are $\boxplus$-convolution powers of the same $\mu$; 
see Proposition \ref{prop:5.3} below.

\vspace{6pt}

$3^o$ One can rewrite Equation (\ref{eqn:1.11}) as 
\begin{equation}  \label{eqn:1.24}
R_{\mu \boxright \nu}  \cdot \bigl( 1+M_{\nu} \bigr) =
R_{\mu} \Bigl( \, z_1 (1+M_{\nu}), \ldots , z_k (1+M_{\nu}) 
\, \Bigr) ,
\end{equation}
and then one can equate coefficients in the series on the two sides
of (\ref{eqn:1.24}), in order to obtain an explicit combinatorial 
formula for the coefficients of $R_{\mu \boxright \nu}$. This 
in turn can be used to obtain an explicit formula for the 
moments of $\mu \boxright \nu$, which is stated next. In Theorem
\ref{thm:1.3}, $NC(n)$ is the set of non-crossing partitions of 
$\{ 1, \ldots , n \}$ (cf. review of $NC(n)$ terminology 
in Section 2B of the paper). The notation 
``$( i_1, \ldots , i_n ) \mid V$'' stands for
``$( i_{v(1)}, \ldots , i_{v(p)} )$'', where 
$V = \{ v(1), \ldots , v(p) \}$ is a non-empty subset of 
$\{ 1, \ldots , n \}$ (listed with $v(1) < \cdots < v(p)$) and
$i_1, \ldots , i_n$ are some indices in $\{ 1, \ldots , k \}$.
\end{remark}

\begin{theorem}  \label{thm:1.3}
Let $\mu, \nu$ be distributions in $\dalg (k)$. For every $n \geq 1$
and $1 \leq i_1, \ldots , i_n \leq k$ let us denote the coefficients
of $z_{i_1} \cdots z_{i_n}$ in the series $R_{\mu}$ and $R_{\nu}$ 
by $\alpha_{(i_1, \ldots , i_n)}$ and $\beta_{(i_1, \ldots , i_n)}$,
respectively. Then for every $n \geq 1$ and 
$1 \leq i_1, \ldots , i_n \leq k$ one has
\begin{equation}  \label{eqn:1.31}
( \mu \boxright \nu ) ( X_{i_1} \cdots X_{i_n} ) = 
\end{equation}
\[
\sum_{\pi \in NC(n)} \ 
\Bigl( \, \prod_{  \begin{array}{c}
{\scriptstyle V \ outer }                 \\
{\scriptstyle block \ of \ \pi}
\end{array}  } \ \ 
\alpha_{ ( i_1, \ldots , i_n ) |V } \, \Bigr) \cdot
\Bigl( \, \prod_{  \begin{array}{c}
{\scriptstyle W \ inner }                 \\
{\scriptstyle block \ of \ \pi}
\end{array}  } \ \ 
\alpha_{ ( i_1, \ldots , i_n ) |W } +
\beta_{ ( i_1, \ldots , i_n ) |W } \, \Bigr) .
\]
\end{theorem}

Based on the moment formula from Theorem \ref{thm:1.3} one can
find an ``operator model on the full Fock space'' for 
$\boxright$. This is a recipe which starts from the data stored in 
the $R$-transforms $R_{\mu}$ and $R_{\nu}$, and uses creation and 
annihilation operators on the full Fock space over $\bC^{2k}$ in 
order to produce a $k$-tuple of operators with distribution 
$\mu \boxright \nu$. The precise description of how this works
appears in Theorem \ref{thm:4.4} of the paper. Once the full 
Fock space model is in place it is easy to see that one can in 
fact upgrade it to a more general operator model for 
$\boxright$, not making specific reference to the full Fock space,
and described as follows.

\begin{theorem}   \label{thm:1.4} 
Let $\cH$ be a Hilbert space, let $\Omega$ be a unit-vector in 
$\cH$, and let $\varphi$ be the vector-state defined by $\Omega$
on $B( \cH )$. Suppose that 
$A_1, \ldots , A_k, B_1, \ldots , B_k \in B( \cH )$ are such that
$\{ A_1, \ldots , A_k \}$ is free from $\{ B_1, \ldots , B_k \}$ 
in $( \, B( \cH ), \varphi \, )$, and let $\mu , \nu$ denote the 
joint distributions of the $k$-tuples $A_1, \ldots , A_k$ and 
respectively $B_1, \ldots , B_k$. Let moreover $P \in B( \cH )$ 
denote the orthogonal projection onto the 1-dimensional subspace 
$\bC \Omega$ of $\cH$, and consider the operators
\begin{equation}    \label{eqn:1.41}
C_i := A_i + (1-P) B_i \, (1-P) \in B( \cH ), \ \ 
1 \leq i \leq k.
\end{equation}
Then the joint distribution of $C_1, \ldots , C_k$ with respect
to $\varphi$ is equal to $\mu \boxright \nu$.
\end{theorem}

\vspace{6pt}

Now, any given pair of distributions 
$\mu, \nu \in \cD_c (k)$ can be made to appear in the setting of
Theorem \ref{thm:1.4}, in such a way that the operators 
$A_1, \ldots , A_k$, 
$B_1, \ldots , B_k$ involved in the theorem are all selfadjoint. 
(This is done via a standard free product construction -- cf.
Remark \ref{rem:4.10} below.)
Since in this case the operators $C_1, \ldots , C_k$ 
from Equation (\ref{eqn:1.41}) are selfadjoint as well, one thus 
obtains the following corollary, giving the desired fact that 
$\boxright$ can be defined as a binary operation on $\cD_c (k)$.

\begin{corollary}  \label{cor:1.6}
If $\mu, \nu$ are in $\cD_c (k)$ then $\mu \boxright \nu$ belongs
to $\cD_c (k)$ as well.
\end{corollary}

\begin{remark}   \label{rem:1.5}
In the 1-variable framework, the study of $\boxright$ was 
started in \cite{L07}. That paper gives an operator model for 
$\mu \boxright \nu$ obtained via an ``$s$-free product'' 
construction for Hilbert spaces, and where $\mu \boxright \nu$ 
appears as the distribution of the sum of two ``$s$-free operators'' 
with distributions $\mu$ and $\nu$, respectively.
By using Theorem \ref{thm:1.4}, it is easy to find a $k$-variable 
analogue for this fact -- that is, one can make $\mu \boxright \nu$ 
appear as the distribution of the sum of two $s$-free $k$-tuples on an 
$s$-free product Hilbert space. The way how this is done is outlined
in Remark \ref{rem:4.11} below. 
\end{remark}

$\ $

The next part of the introduction (from Remark \ref{rem:1.7} to 
Proposition \ref{prop:1.10}) explains how $\boxright$ relates to
the work in \cite{BN09} concerning evolution towards 
$\boxplus$-infinite divisibility and its connection to the 
free Brownian motion.

\begin{remark}   \label{rem:1.7}
Here is a brief summary of relevant results from \cite{BN09}. One
considers a family of bijective transformations 
$\{ \bB_t \mid t \geq 0 \}$ of $\dalg (k)$ defined by 
\[
\bB_t ( \mu ) = \Bigl( \, \mu^{\boxplus (1+t)} \, 
                \Bigr)^{\uplus 1/(1+t)}, \ \ 
\forall \, t \geq 0, \ \forall \, \mu \in \dalg (k),
\] 
where the $\boxplus$-powers and $\uplus$-powers are taken in
connection to free and respectively Boolean convolution. The
transformations $\bB_t$ form a semigroup 
($\bB_{s+t} = \bB_s \circ \bB_t$, $\forall \, s,t \geq 0$), each 
of them carries $\cD_c (k)$ into itself, and at $t=1$ one has 
$\bB_1 = \bB$, the Boolean Bercovici-Pata bijection that was also 
encountered in Remark \ref{rem:1.2}.2. Thus for a fixed 
$\mu \in \cD_c (k)$ the process $\{ \bB_t ( \mu ) \mid t \geq 0 \}$ 
can be viewed as some kind of 
``evolution of $\mu$ towards $\boxplus$-infinite divisibility''
(since $\bB_t ( \mu )$ is infinitely divisible for all $t \geq 1$).

On the other hand let us recall that the free Brownian motion 
started at a distribution $\nu \in \cD_c (k)$ is the process
$\{ \nu \boxplus \gamma^{\boxplus t} \mid t \geq 0 \}$, where 
$\gamma \in \cD_c (k)$ is the joint distribution of a standard 
semicircular system (a free family of $k$ centered semicircular 
elements of variance 1). The paper \cite{BN09} puts into 
evidence a certain transformation 
$\Phi : \dalg (k) \to \dalg (k)$ which carries $\cD_c (k)$ into 
itself and has the property that 
\begin{equation}     \label{eqn:1.71}
\Phi ( \nu \boxplus \gamma^{\boxplus t} ) = \bB_t \bigl( \, 
\Phi ( \nu ) \, \bigr), \ \ \forall \, \nu \in \dalg (k), \
\forall \, t \geq 0.
\end{equation}
In other words, (\ref{eqn:1.71}) says that a relation of the form
``$\Phi ( \nu ) = \mu$'' is not affected when $\nu$ evolves under 
the free Brownian motion while $\mu$ evolves under the action of
the semigroup $\{ \bB_t \mid t \geq 0 \}$. The transformation 
$\Phi$ from \cite{BN09} turns out to be related to subordination 
distributions, as follows.
\end{remark}

\begin{theorem}   \label{thm:1.8}
For every distribution $\nu \in \dalg (k)$ one has that 
\begin{equation}    \label{eqn:1.81}
\gamma \boxright \nu = \bB \bigl( \, \Phi ( \nu ) \, \bigr),
\end{equation}
where $\gamma$ is as above (the joint distribution of a standard 
semicircular system) and $\bB$ is the Boolean Bercovici-Pata 
bijection.
\end{theorem}

\begin{remark}   \label{rem:1.9}

$1^o$ Equation (\ref{eqn:1.81}) thus offers an alternative 
description for $\Phi$:
\begin{equation}    \label{eqn:1.91}
\Phi ( \nu ) = \bB^{-1} \bigl( \, \gamma \boxright \nu \, \bigr),
\ \ \nu \in \dalg (k).
\end{equation}
It is worth noting that the two main properties of $\Phi$ obtained 
in \cite{BN09} (formula (\ref{eqn:1.71}) and the fact that $\Phi$ 
maps $\cD_c (k)$ into itself) are very easy to derive by starting 
from this description and by invoking the suitable properties of 
subordination distributions; see Proposition \ref{prop:5.6}.

$2^o$ It is also worth noting that one has a simple explicit 
formula for how $\mu \boxright \nu$ itself evolves under the action 
of the $\bB_t$. This formula pops up when one compares the explicit 
descriptions for the free and the Boolean cumulants of 
$\mu \boxright \nu$ (see Remark \ref{rem:3.8}.1 and Proposition 
\ref{prop:5.1} below), and is described as follows.
\end{remark}

\begin{proposition}    \label{prop:1.10}
Let $\mu , \nu$ be two distributions in $\dalg (k)$. Then for 
every $t \geq 0$ one has:
\begin{equation}   \label{eqn:1.101}  
\bB_t ( \mu \boxright \nu ) = \mu \boxright 
\bigl( \mu^{\boxplus t} \boxplus \nu \bigr).
\end{equation}
\end{proposition}

$\ $

The final part of the introduction discusses two other interesting 
algebraic properties of $\boxright$, obtained by extrapolating 
functional equations which are known to be satisfied by 
subordination functions in the 1-variable framework. One of these 
two properties extends a remarkable formula for 
the sum of the subordination functions of $\mu \boxplus \nu$ with 
respect to $\mu$ and to $\nu$ (see e.g. Theorem 4.1 in 
\cite{BB07}). This formula can be equivalently written in terms 
of the $\eta$-series of $\mu \boxright \nu$ and 
$\nu \boxright \mu$, and in this form 
it goes through to the multi-variable framework, as follows.

\begin{proposition}   \label{prop:1.11}
One has that 
\begin{equation}   \label{eqn:1.111}
\eta_{\mu \boxright \nu} + \eta_{\nu \boxright \mu}
= \eta_{\mu \boxplus \nu}, \ \ \forall \, \mu , \nu \in \dalg (k).
\end{equation} 
\end{proposition}

\vspace{6pt}

Another property of $\boxright$ comes from the functional equation 
satisfied by the subordination function of a convolution power 
$\nu^{\boxplus p}$ with respect to $\nu$, where $\nu$ is a 
probability measure on $\bR$ and $p \in [ 1, \infty )$ (see Theorem 
2.5 in \cite{BB04}). This too can be translated into a formula
involving $\eta$-series, which goes through to multi-variable 
framework. More precisely, the subordination 
distribution of $\nu^{\boxplus p}$ with 
respect to $\nu$ can be considered for any $\nu \in \dalg (k)$ and 
$p \in [ 1, \infty )$ (see Definition \ref{def:6.3} below), and the 
following statement holds.

\begin{proposition}    \label{prop:1.12}
For every $\nu \in \dalg (k)$ and $p \geq 1$, the
subordination distribution of $\nu^{\boxplus p}$ with respect 
to $\nu$ is equal to $\bigl( \bB ( \nu ) \bigr)^{\boxplus (p-1)}$.
\end{proposition}

In particular, for distributions in $\cD_c (k)$ one gets the 
following corollary.

\begin{corollary}    \label{cor:1.13}
Let $\nu$ be a distribution in $\cD_c (k)$. Then for 
every $p \geq 1$ the subordination distribution of 
$\nu^{\boxplus p}$ with respect to $\nu$ is still in $\cD_c (k)$,
and is $\boxplus$-infinitely divisible.
\end{corollary}

One can also put into evidence other natural situations when 
subordination distributions in $\cD_c (k)$ are sure to be 
$\boxplus$-infinitely divisible. In particular, an immediate 
consequence of Remark \ref{rem:1.2}.1 (combined with Corollary
\ref{cor:1.6}) is that $\mu \boxright \nu$ is $\boxplus$-infinitely
divisible whenever $\mu , \nu \in \cD_c (k)$ and $\mu$ is itself 
$\boxplus$-infinitely divisible; see Corollary \ref{cor:4.11} below.

$\ $

\begin{remark}      \label{rem:1.14}
After circulating the first version of this paper, I was made aware
of the connection between the results obtained here and the paper
\cite{A08} of Anshelevich, where a two-variable extension of the 
transformation $\Phi$ from Remark \ref{rem:1.7} is being studied. 
More precisely, \cite{A08} introduces a map
\[
\dalg (k) \times \dalg (k) \ni ( \rho , \psi ) \mapsto 
\Phi [ \rho , \psi ] \in \dalg (k)
\]
with the property that $\Phi [ \rho , \psi ] \in \cD_c (k)$ for every
$\rho , \psi \in \cD_c (k)$ such that $\rho$ is $\boxplus$-infinitely 
divisible, and with the property that 
\begin{equation}   \label{eqn:1.141}
\Phi [ \gamma , \psi ] = \Phi ( \psi ), \ \ \
\forall \, \psi \in \dalg (k)
\end{equation}
(where $\gamma \in \cD_c (k)$ is the same as in Remark \ref{rem:1.7}). 
The formula which gives the translation 
between the results of the present paper and those of \cite{A08} is
\begin{equation}   \label{eqn:1.142}
\Phi [ \rho , \psi ] = \bB^{-1} \bigl( \, \rho \boxright \psi
\, \bigr), \ \ \ \forall \, \rho , \psi \in \dalg (k).
\end{equation}
Equation (\ref{eqn:1.142}) can be used to explain why the argument
$\rho$ in $\Phi [ \rho, \psi ]$ is naturally chosen to be 
$\boxplus$-infinitely divisible: as observed right before the 
present remark, one has in this situation that
$\rho \boxright \psi$ is $\boxplus$-infinitely divisible, hence 
that $\bB^{-1} ( \rho \boxright \psi ) \in \cD_c (k)$ for every 
$\psi \in \cD_c (k)$. (Another explanation for why $\rho$ is 
naturally taken to be infinitely divisible is presented in 
Remark 10 of \cite{A08}.)

By using the formula (\ref{eqn:1.142}), the description of $\Phi$ 
from the above Remark \ref{rem:1.9}.1 is reduced to 
(\ref{eqn:1.141}), and it is also easily seen that Proposition 
\ref{prop:1.10} of the present 
paper is equivalent to Theorem 11(b) from \cite{A08}.

The scope of \cite{A08} is different from (albeit overlapping with)
the one of the present paper, and the methods of proof are different,
invoking e.g. results about conditionally positive definite 
functionals, or about a multi-variable version of monotonic 
convolution -- see specifics in Section 4.2 of \cite{A08}.
\end{remark}

\begin{remark}      \label{rem:1.15}
{\em (Organization of the paper.) }
Besides the introduction, the paper has five other sections.
Section 2 contains a review of some background and notations.
Section 3 derives explicit combinatorial formulas for the free and 
Boolean cumulants of $\mu \boxright \nu$, and uses them in order to
obtain the moment formula announced in Theorem \ref{thm:1.3}. 
Section 4 is devoted to operator models, and to the proof of 
Theorem \ref{thm:1.4}.
Section 5 discusses in more detail the relations to the 
transformations $\bB_t$ that were stated in Theorem \ref{thm:1.8} 
and Proposition \ref{prop:1.10}.
Finally, Section 6 discusses in more detail the statements made
in Propositions \ref{prop:1.11}, \ref{prop:1.12}, and in 
Corollary \ref{cor:1.13}.
\end{remark}

\vspace{10pt}

\begin{center}
{\bf\large 2. Background and notations}
\end{center}
\setcounter{section}{2}
\setcounter{equation}{0}
\setcounter{theorem}{0}

\vspace{4pt}

\begin{center}
{\bf 2A. Motivation from 1-variable framework}
\end{center}

\begin{remark}   \label{rem:2.1}
Recall that for a probability distribution $\mu$ on $\bR$, the 
{\em Cauchy transform} of $\mu$ is the 
analytic function $G_{\mu}$ defined by
\begin{equation}   \label{eqn:2.11}
G_{\mu} (z) = \int_{\bR} \frac{1}{z-t} \ d \mu (t), \ \ 
z \in \bC \setminus \bR .
\end{equation}
The {\em reciprocal Cauchy transform} $F_{\mu}$ is defined by 
\begin{equation}   \label{eqn:2.12}
F_{\mu} (z) = 1/ G_{\mu} (z), \ \ 
z \in \bC \setminus \bR .
\end{equation}
It is easily checked that $G_{\mu}$ maps the upper half-plane
$\bC^{+} = \{ z \in \bC \mid \mbox{Im} (z) >0 \}$ to the lower 
half-plane $\bC^{-} = \{ z \in \bC \mid \mbox{Im} (z) <0 \}$;
as a consequence of this, $F_{\mu}$ can be viewed as an analytic 
self-map of $\bC^{+}$.
The measure $\mu$ is uniquely determined by $G_{\mu}$ (hence by
$F_{\mu}$ as well); and more precisely, $\mu$ can be retrieved 
from $G_{\mu}$ by a procedure called ``Stieltjes inversion 
formula'' (see e.g. \cite{A65}).

Let $\fF$ denote the set of all analytic self-maps of $\bC^{+}$ 
that can arise as $F_{\mu}$ for some probability measure $\mu$ on
$\bR$. One has a very nice intrinsic description of $\fF$, that
\begin{equation}    \label{eqn:2.13}
\fF = \Bigl\{ F: \bC^{+} \to \bC^{+} \mid F \mbox{ is analytic and }
\lim_{t \to \infty} \frac{F(it)}{it} =1  \Bigr\} .
\end{equation}
(For a nice review of this and other properties of $\fF$ one can 
consult Section 2 of \cite{M92} or Section 5 of \cite{BV93}.)

As mentioned in  the introduction, the starting point of this paper
is that for any two probability measures $\mu, \nu$ on $\bR$, there
exists a {\em subordination function} $\omega \in \fF$ 
such that
\begin{equation}    \label{eqn:2.14}
G_{\mu \boxplus \nu} (z) = G_{\nu} \bigl( \, \omega (z) \, \bigr),
\ \ z \in \bC^{+}.
\end{equation}
With $\mu, \nu, \omega$ as in (\ref{eqn:2.14}), it is natural to 
consider the unique probability measure $\sigma$ on $\bR$ such that
$F_{\sigma} = \omega$. This $\sigma$ was studied in \cite{L07}, 
where it is called the {\em $s$-free convolution} of $\mu$ and 
$\nu$, and is denoted by $\mu \boxright \nu$. The name ``$s$-free 
convolution'' appears in \cite{L07} in connection to a 
suitably tailored concept of ``$s$-freeness'' that is also
introduced in \cite{L07}. Since $s$-freeness is only marginally 
addressed in the present paper, we will refer to $\mu \boxright \nu$ 
by just calling it the {\em subordination distribution of 
$\mu \boxplus \nu$ with respect to $\nu$}. We will only look at 
$\mu \boxright \nu$ in the special case when $\mu$ and $\nu$ are 
compactly supported; in this case $\mu \boxright \nu$ is 
compactly supported as well (as one sees by examining the operator
model obtained for $\mu \boxright \nu$ in \cite{L07}).
\end{remark}

\begin{remark}   \label{rem:2.2}
If $\mu$ is a compactly supported probability measure on $\bR$, 
then in particular $\mu$ has moments of all orders:
\[
m_n := \int_{- \infty}^{\infty} t^n \ d \mu (t), \ \ 
n \in \bN,
\]
and one can form the {\em moment series} of $\mu$,
\begin{equation}    \label{eqn:2.21}
M_{\mu} (z) = \sum_{n=1}^{\infty} m_n z^n.
\end{equation}
In (\ref{eqn:2.21}), $M_{\mu}$ can be viewed as an analytic function
on a neighbourhood of 0, but in the present paper it is preferable 
to treat it as a formal power series in $z$. It is immediate that 
$M_{\mu}$ is connected to the Cauchy transform $G_{\mu}$ by the formula
\begin{equation}    \label{eqn:2.22}
1 + M_{\mu} (1/z) =  z \cdot G_{\mu} (z),
\end{equation}
where in (\ref{eqn:2.22}) it is convenient to also treat $G_{\mu}$ 
as a power series (obtained by writing 
$1/(t-z) = \sum_{n=1}^{\infty} t^{n-1}/z^n$ and then integrating
term by term on the right-hand side of (\ref{eqn:2.11}), for 
$z \in \bC^{+}$ with $|z|$ large enough).

In the study of free additive convolution a fundamental object
is Voiculescu's $R$-transform, which has the linearizing property 
that $R_{\mu \boxplus \nu} = R_{\mu} + R_{\nu}$. For a compactly
supported probability measure $\mu$ on $\bR$, the $R$-transform 
$R_{\mu}$ can be viewed as a power series, defined in terms 
of $M_{\mu}$ as the unique solution of the equation 
\begin{equation}    \label{eqn:2.23}
R_{\mu} \Bigl( \, z( 1+ M_{\mu} (z) \, \Bigr) =  M_{\mu} (z)
\end{equation}
(equation in $\bC [[z]]$, where $M_{\mu}$ is given as data and 
$R_{\mu}$ is the unknown). For the next proposition it is more 
convenient to write the definition of $R_{\mu}$ by emphasizing 
its relation to the Cauchy transform $G_{\mu}$. On these lines 
one first defines the so-called {\em $K$-transform} of $\mu$, 
which is simply the inverse under composition
\begin{equation}    \label{eqn:2.24}
K_{\mu} := G_{\mu}^{\langle -1 \rangle}.
\end{equation}
$K_{\mu}$ is a Laurent series of the form $K_{\mu} (z)$
 = $\frac{1}{z} + \alpha_1 + \alpha_2 z + \alpha_3 z^2 + \cdots$,
and one has 
\footnote{The original definition of the $R$-transform, made in 
\cite{V86}, simply has $\cR_{\mu} (z) = K_{\mu} (z) - 1/z$. The 
present paper uses the shifted version 
$R_{\mu} (z) = z \cR_{\mu} (z)$, which is more convenient for 
extension to a multi-variable framework.}
\begin{equation}    \label{eqn:2.25}
R_{\mu} (z) = z \Bigl( \, K_{\mu} (z) - \frac{1}{z} \, \Bigr).
\end{equation}
In the next proposition, Equation (\ref{eqn:2.25}) will be used 
in the equivalent form giving $K_{\mu}$ in terms of $R_{\mu}$,
\begin{equation}  \label{eqn:2.26}
K_{\mu} (z) = \frac{1 + R_{\mu} (z)}{z}.
\end{equation}
\end{remark}

\vspace{10pt}

\begin{proposition}   \label{prop:2.3}
Let $\mu, \nu$ be compactly supported probability 
measures on $\bR$, and let the probability measure 
$\mu \boxright \nu$ be defined as in Remark \ref{rem:2.1}.
Then
\begin{equation}  \label{eqn:2.31}
R_{\mu \boxright \nu} (z) =  \frac{ R_{\mu} 
       \bigl( \ z( 1+ M_{\nu} (z)) \ \bigr) }{ 1+ M_{\nu} (z) } .
\end{equation}
\end{proposition}

\begin{proof}
Let us denote for brevity $\mu \boxright \nu =: \sigma$. From how
$\mu \boxright \nu$ is defined we have that 
\begin{equation}  \label{eqn:2.32}
G_{\mu \boxplus \nu} = G_{\nu} \circ F_{\sigma}.
\end{equation}
By taking inverses under composition on both sides of 
(\ref{eqn:2.32}) one finds that $K_{\mu \boxplus \nu}$ = 
$F_{\sigma}^{\langle -1 \rangle} \circ K_{\nu}$, hence that 
$F_{\sigma} \circ K_{\mu \boxplus \nu} = K_{\nu}$;
this in in turn implies that 
$G_{\sigma} \circ K_{\mu \boxplus \nu} = 1/K_{\nu}$,
and that
$K_{\mu \boxplus \nu} = K_{\sigma} \circ \bigl( 1/K_{\nu} \bigr)$.
So one gets the formula:
\begin{equation}  \label{eqn:2.33}
K_{\mu \boxplus \nu} (w) = K_{\sigma}
\Bigl( \, 1/K_{\nu} (w) \, \Bigr)  
\end{equation}
(equality of Laurent series in an indeterminate $w$).

In (\ref{eqn:2.33}) let us next replace the $K$-transforms of 
$\mu \boxplus \nu$ and of $\sigma$ in terms of the corresponding 
$R$-transforms, by using Equation (\ref{eqn:2.26}). On the 
left-hand side we obtain
\[
K_{\mu \boxplus \nu} (w) 
= \frac{1 + R_{\mu \boxplus \nu} (w) }{w} 
= \frac{1 + R_{\mu} (w) + R_{\nu} (w)}{w} 
= \frac{R_{\mu} (w)}{w} + K_{\nu} (w),
\]
while on the right-hand side we obtain
\[
K_{\sigma} \Bigl( \, 1/K_{\nu} (w) \, \Bigr)  
= \frac{ 1+ R_{\sigma} \bigl( \, 1/K_{\nu} (w) \, \bigr)}{
    1/K_{\nu} (w)  }
=  K_{\nu} (w) + K_{\nu} (w) \cdot 
R_{\sigma} \bigl( \, 1/K_{\nu} (w) \, \bigr).
\]
After making these replacements and after subtracting $K_{\nu} (w)$ 
out of both sides of (\ref{eqn:2.33}) one arrives to
\begin{equation}  \label{eqn:2.34}
\frac{R_{\mu} (w)}{w} = K_{\nu} (w) \cdot R_{\sigma} 
\bigl( \, 1/K_{\nu} (w) \, \bigr).
\end{equation}

Finally, in (\ref{eqn:2.34}) let us make the substitution 
$z = 1/K_{\nu} (w)$, with inverse 
$w = G_{\nu} (1/z) = z( 1+ M_{\nu} (z) )$; this substitution 
converts (\ref{eqn:2.34}) into
\[
\frac{ R_{\mu} \bigl( \, z( 1+ M_{\nu} (z)) \, \bigr) }{ 
z (1+ M_{\nu} (z)) } =
\frac{1}{z} \cdot R_{\sigma} (z),
\]
and (\ref{eqn:2.31}) follows.
\end{proof}

$\ $

\begin{center}
{\bf 2B. Non-crossing partitions}
\end{center}

\begin{notation}         \label{def:2.4} 
{\em ($NC(n)$ terminology.) }
Let $n$ be a positive integer.

$1^o$ Let $\pi$ = $\{ V_{1} , \ldots , V_{p} \}$ be a partition 
of $\{ 1, \ldots ,n \}$ -- i.e. $V_{1} , \ldots , V_{p}$ are 
pairwise disjoint non-empty sets (called the {\em blocks} of $\pi$), 
and $V_{1} \cup \cdots \cup V_{p}$ = 
$\{ 1, \ldots , n \}$. We say that $\pi$ is {\em non-crossing} if for 
every $1 \leq i < j < k < \ell \leq n$ such that $i$ is in the same
block with $k$ and $j$ is in the same block with $\ell$, it necessarily
follows that all of $i,j,k, \ell$ are in the same block of $\pi$.
The set of all non-crossing partitions of 
$\{ 1, \ldots , n \}$ will be denoted by $NC(n).$ 

$2^o$ Let $\pi$ be a partition in $NC(n)$. Since $\pi$ is, after 
all, a set of subsets of $\{ 1, \ldots , n \}$, it will be convenient
to write ``$V \in \pi$'' as a shorthand for ``$V$ is a block of 
$\pi$''. In the same vein, various calculations throughout the paper
will use functions ``$c: \pi \to \{ 1,2 \}$''. Such a function is 
thus a recipe for assigning a number $c(V) \in \{ 1,2 \}$ to 
every block $V$ of $\pi$, and will be referred to as a 
{\em colouring} of $\pi$.

$3^o$ For $\pi \in NC(n)$, the number of blocks of $\pi$ will be 
denoted by $| \pi |$. 

$4^o$ Let $\pi$ be a partition in $NC(n)$, and let $V$ be a block
of $\pi$. If there exists a block $W$ of $\pi$ such that 
$\min (W) < \min (V)$ and $\max (W) > \max (V)$, then one says 
that $V$ is an {\em inner} block of $\pi$. In the opposite case 
one says that $V$ is an {\em outer} block of $\pi$.

$5^o$ Every partition $\pi \in NC(n)$ has a special colouring 
$\oo_{\pi} : \pi \to \{ 1, 2 \}$ which will be called the 
{\em inner/outer colouring} of $\pi$, and is defined by
\begin{equation}   \label{eqn:3.42}
\oo_{\pi} (V) = \left\{ \begin{array}{cl}
1 & \mbox{if $V$ is outer}   \\ 
2 & \mbox{if $V$ is inner,}   
\end{array}  \right. \ \ \ V \in \pi .
\end{equation}
\end{notation}

\begin{remark}     \label{rem:2.5}
$NC(n)$ is partially ordered by {\em reverse refinement}: for 
$\pi , \rho \in NC(n)$ one writes
``$\pi \leq \rho$'' to mean that every block of $\rho$ is a union of
blocks of $\pi$. The minimal and maximal element of $( NC(n), \leq )$ 
are denoted by $0_n$ (the partition of $\{ 1, \ldots , n \}$ into 
$n$ singleton blocks) and respectively $1_n$ (the partition of 
$\{ 1, \ldots , n \}$ into only one block).  

Let $\rho = \{ W_1 , \ldots , W_q \}$ be a fixed partition 
in $NC(n)$.  It is easy to see that one has a
natural poset isomorphism
\begin{equation}   \label{eqn:2.51}
\{ \pi \in NC(n) \mid \pi \leq \rho \} \ni \pi \mapsto
( \pi_1, \ldots , \pi_q ) \in
NC( \, |W_1| \, ) \times \cdots \times NC( \, |W_q| \, )
\end{equation}
where for every $1 \leq j \leq q$ the partition
$\pi_j \in NC( |W_j| )$ is obtained by restricting $\pi$ to $W_j$
and by re-denoting the elements of $W_j$, in increasing order, so that
they become $1,2, \ldots , |W_j|$. 
This is a particular case of a more general factorization 
property satisfied by the intervals of the poset 
$( NC(n) , \leq )$ -- see Lecture 9 in \cite{NS06}.
\end{remark}

\begin{remark}     \label{rem:2.6}
This paper also makes use of an other partial order relation on 
$NC(n)$, which was introduced in \cite{BN08} and is denoted by 
``$\leqleq$''. For $\pi , \rho \in NC(n)$ one writes
``$\pi \leqleq \rho$'' to mean that $\pi \leq \rho$ and that, in 
addition, the following condition is fulfilled:
\begin{equation}  
\left\{  \begin{array}{l}
\mbox{For every block $W$ of $\rho$ there exists a block}  \\
\mbox{$V$ of $\pi$ such that $\min (W), \max (W) \in V$.}
\end{array}  \right.
\end{equation}
It is immediately verified that ``$\leqleq$'' is indeed a partial 
order relation on $NC(n)$. It is much coarser than the reversed 
refinement order. For instance, the inequality $\pi \leqleq 1_n$ is 
not holding for all $\pi \in NC(n)$, but it rather amounts to the 
condition that the numbers $1$ and $n$ belong to the same block 
of $\pi$ (or equivalently, that $\pi$ has a unique outer block).
At the other end of $NC(n)$, the inequality $\pi \geqgeq 0_n$ can 
only take place when $\pi = 0_n$. The remaining part of Section 2B
reviews a couple of other properties of $\leqleq$ that will be 
used later on in the paper.
\end{remark}

\begin{definition}   \label{def:2.7}
Let $\pi, \rho$ be partitions in $NC(n)$ such that 
$\pi \leqleq \rho$. A block $V$ of $\pi$ is said to be 
{\em $\rho$-special} when there exists
a block $W$ of $\rho$ such that $\min (V) = \min (W)$ and 
$\max (V) = \max (W)$. 
\end{definition}

\begin{proposition}   \label{prop:2.8}
Let $\pi \in NC(n)$ be such that $\pi \leqleq 1_n$, and consider 
the set of partitions
\begin{equation}   \label{eqn:2.81}
\{ \rho \in NC(n) \mid \pi \leqleq \rho \leqleq 1_n \} .
\end{equation}
Then 
$\rho \mapsto \{ V \in \pi \mid V \mbox{ is $\rho$-special} \}$
is a one-to-one map from the set (\ref{eqn:2.81}) to the set of 
subsets of $\pi$. 
\footnote{According to the conventions made in Notation 
\ref{def:2.4}.2, ``subset of $\pi$'' stands here for
``set of blocks of $\pi$''.}
The image of this map is equal to
$\{ \fV \subseteq \pi \mid \fV \ni V_0 \}$, where 
$V_0$ denotes the unique outer block of $\pi$. 
\end{proposition}

\vspace{10pt}

\noindent
For the proof of Proposition \ref{prop:2.8}, the reader is 
referred to Proposition 2.13 and Remark 2.14 of \cite{BN08}.

$\ $

\begin{remark}    \label{rem:2.9}
{\em (Interval partitions.)}
A partition $\pi$ of $\{ 1, \ldots , n \}$ is said to be an 
{\em interval partition} if every block $V$ of $\pi$ is of 
the form $V = [i,j] \cap \bZ$ for some $1 \leq i \leq j \leq n$.
The set of all interval partitions of $\{ 1, \ldots , n \}$ will 
be denoted by $\Int (n)$. It is clear that $\Int (n) \subseteq NC(n)$,
and it is easily verified that every interval partition is a 
maximal element of the poset $( NC(n), \leqleq )$. It is moreover
easy to see (left as exercise to the reader) that for every 
$\pi \in NC(n)$ there exists a unique $\rho \in \Int (n)$ such that 
$\pi \leqleq \rho$; the blocks of this special interval partition 
$\rho$ are in some sense the ``convex hulls'' of the outer blocks 
of $\pi$.
\end{remark}

$\ $

\begin{center}
{\bf 2C. Power series in $k$ noncommuting indeterminates}
\end{center}

\begin{notation}   \label{def:2.10}
We will denote by
$\bC \langle \langle z_1, \ldots , z_k \rangle \rangle$
the set of power series with complex coefficients 
in the non-commuting indeterminates $z_1, \ldots , z_k$, 
and we will use the notation 
$\bC_0 \langle \langle z_1, \ldots , z_k \rangle \rangle$
for the set of series in 
$\bC \langle \langle z_1, \ldots , z_k \rangle \rangle$
which have vanishing constant term. The general form of a series
$f \in \bC_0 \langle \langle z_1, \ldots , z_k \rangle \rangle$ 
is thus
\begin{equation}     \label{eqn:2.101}
f( z_1, \ldots , z_k) = \sum_{n=1}^{\infty} \
\sum_{i_1, \ldots , i_n =1}^k  \ \alpha_{(i_1, \ldots , i_n)}
z_{i_1} \cdots z_{i_n} 
\end{equation}
where the coefficients  $\alpha_{ (i_1, \ldots , i_n) }$ are 
from $\bC$.
\end{notation}

\begin{definition}    \label{def:2.11}
{\em (Coefficients for series in $\ncserk$.)}

$1^o$ For $n \geq 1$ and $1 \leq i_1, \ldots , i_n \leq k$ we will 
denote by
\begin{equation}  \label{eqn:2.111}
\cf_{(i_1, \ldots , i_n)} :
\bC_0 \langle \langle z_1, \ldots , z_k \rangle \rangle \to \bC
\end{equation}
the linear functional which extracts the coefficient of 
$z_{i_1} \cdots z_{i_n}$ in a series 
$f \in \bC_0 \langle \langle z_1, \ldots , z_k \rangle \rangle$.
Thus for $f$ written as in Equation (\ref{eqn:2.101}) we have 
$\cf_{(i_1, \ldots , i_n)} (f) = \alpha_{(i_1, \ldots , i_n)}$.

$2^o$ Suppose we are given a positive integer $n$, some indices 
$i_1, \ldots , i_n \in \{ 1, \ldots , k \}$, and a partition 
$\pi \in NC(n)$. We define a (generally non-linear) functional 
\begin{equation}  \label{eqn:2.112}
\cf_{(i_1, \ldots , i_n) ; \pi} :
\bC_0 \langle \langle z_1, \ldots , z_k \rangle \rangle \to \bC ,
\end{equation}
as follows. For every block $V = \{ b_1, \ldots , b_m \}$ of $\pi$, 
with $1 \leq b_1 < \cdots < b_m \leq n$, let us use the notation 
\[
(i_1, \ldots , i_n) \vert V \ := \ (i_{b_1}, \ldots , i_{b_m})
\in \{ 1, \ldots , k \}^m.
\]
Then we define
\begin{equation}  \label{eqn:2.113}
\cf_{(i_1, \ldots , i_n); \pi} (f) \ := \
\prod_{V \in \pi} \ \cf_{(i_1, \ldots , i_n)|V} (f),
\ \ \forall \, f \in \ncserk .
\end{equation}
(For example if we had $n=5$ and $\pi = \{ \{ 1,4,5 \} , \{ 2,3 \} \}$, 
and if $i_1, \ldots , i_5$ would be some fixed indices in
$\{ 1, \ldots , k \}$, then the above formula would become
\[
\cf_{(i_1, i_2, i_3, i_4, i_5) ; \pi } (f) \ = \
\cf_{(i_1, i_4, i_5)} (f) \cdot
\cf_{(i_2, i_3)} (f),
\]
$f \in \bC_0 \langle \langle z_1, \ldots , z_k \rangle \rangle$.) 
The quantities $\cf_{(i_1, \ldots , i_n); \pi} (f)$ will be 
referred to as {\em generalized coefficients} of the series $f$. 

$3^o$ Suppose that the positive integer $n$, the indices 
$i_1, \ldots , i_n \in \{ 1, \ldots , k \}$ and the partition 
$\pi \in NC(n)$ are as above, and that in addition we are also
given a colouring  $c: \pi \to \{ 1, 2 \}$. Then for any two 
series $f_1,f_2 \in \ncserk$ we define their 
{\em mixed generalized coefficient} corresponding to 
$(i_1, \ldots , i_n)$, $\pi$ and $c$ via the formula
\begin{equation}    \label{eqn:2.114}
\cf_{ (i_1, \ldots , i_n); \pi ; c } (f_1, f_2) :=
\prod_{ V \in \pi } \, \cf_{ (i_1, \ldots , i_n) |V } (f_{c(V)}).
\end{equation}
(For example if we had $n=5$, $\pi = \{ \{ 1,4,5 \} , \{ 2,3 \} \}$
and $c: \pi \to \{ 1,2 \}$ defined by 
$c( \, \{ 1,4,5 \} \, ) = 1$, $c( \, \{ 2,3 \} \, ) = 2$,then
then (\ref{eqn:2.114}) would become
\[
\cf_{(i_1, i_2, i_3, i_4, i_5) ; \pi } (f) \ = \
\cf_{(i_1, i_4, i_5)} (f_1) \cdot
\cf_{(i_2, i_3)} (f_2),
\]
for $f_1, f_2 \in \ncserk$ and $1 \leq i_1, \ldots , i_5 \leq k$.)
\end{definition}

\begin{remark}   \label{rem:2.12}
It is clear that for every $n \geq 1$, 
$1 \leq i_1, \ldots , i_n \leq k$, $\pi \in NC(n)$ and $f \in \ncserk$ 
one has 
\[
\cf_{ (i_1, \ldots , i_n); \pi ; c } (f, f) 
= \cf_{ (i_1, \ldots , i_n); \pi } (f),
\]
for no matter what colouring $c$ of $\pi$. Let us also record here 
the obvious expansion formula
\begin{equation}   \label{eqn:2.121}
\cf_{ (i_1, \ldots , i_n) ; \pi } (f_1 + f_2) =
\sum_{ c: \pi \to \{1,2 \} } \
\cf_{ (i_1, \ldots , i_n) ; \pi; c } (f_1 , f_2),
\end{equation}
holding for every $n \geq 1$, $1 \leq i_1, \ldots , i_n \leq k$,
$\pi \in NC(n)$, and $f_1, f_2 \in \ncserk$.
\end{remark}

\begin{definition}    \label{def:2.13}
{\em (Review of the series $M_{\mu}$, $R_{\mu}$, $\eta_{\mu}$.) }
Let $\mu$ be a distribution in $\dalg (k)$. 

$1^o$ We will denote by $M_{\mu}$ the series in $\ncserk$ defined by
\begin{equation}  \label{eqn:2.131}
M_{\mu} (z_1, \ldots , z_k) := \sum_{n=1}^{\infty} \
\sum_{i_1, \ldots , i_n =1}^k \ \mu (X_{i_1} \cdots X_{i_n})
\, z_{i_1} \cdots z_{i_n}.
\end{equation}
$M_{\mu}$ is called the {\em moment series} of $\mu$, and
its coefficients (the numbers $\mu ( X_{i_1} \cdots X_{i_n} )$, 
with $n \geq 1$ and $1 \leq i_1, \ldots , i_n \leq k$) are called 
the {\em moments} of $\mu$.

$2^o$ The {\em $\eta$-series} of $\mu$ is
\begin{equation}  \label{eqn:2.132}
\eta_{\mu} := M_{\mu} (1+M_{\mu})^{-1} \in \ncserk ,
\end{equation}
where $(1+ M_{\mu})^{-1}$ is the inverse of $1 + M_{\mu}$ under
multiplication in 
$\bC \langle \langle z_1, \ldots , z_k \rangle \rangle$.
The coefficients of $\eta_{\mu}$ are called the
{\em Boolean cumulants} of $\mu$.

$3^o$ There exists a unique series $R_{\mu} \in \ncserk$ which 
satisfies the functional equation
\begin{equation}  \label{eqn:2.133}
R_{\mu} \Bigl( \, z_1 (1+M_{\mu}), \ldots , z_k (1+M_{\mu} \,
\Bigr) = M_{\mu}.
\end{equation}
Indeed, it is easily seen that Equation (\ref{eqn:2.133}) 
amounts to a recursion which determines uniquely the coefficients
of $R_{\mu}$ in terms of those of $M_{\mu}$. The series $R_{\mu}$
is called the {\em R-transform} of $\mu$, and its coefficients
are called the {\em free cumulants} of $\mu$. (See the discussion
in Lecture 16 of \cite{NS06}, and specifically Theorem 16.15 and
Corollary 16.16 of that lecture.) 
\end{definition}

\begin{remark}    \label{rem:2.14}
It is very useful that one has explicit summation formulas which 
express the moments of a distribution $\mu \in \dalg (k)$ either 
in terms of its free cumulants or in terms of its Boolean cumulants.  
These are sometimes referred to as {\em moment-cumulant} formulas.
They say that for every $n \geq 1$ and 
$1 \leq i_1, \ldots , i_n \leq k$ one has
\begin{equation}    \label{eqn:2.141}
\mu ( X_{i_1} \cdots X_{i_n} ) =
\sum_{\pi \in NC(n)} \cf_{(i_1, \ldots , i_n); \pi} ( R_{\mu} )
\end{equation}
and respectively
\begin{equation}    \label{eqn:2.142}
\mu ( X_{i_1} \cdots X_{i_n} ) =
\sum_{\pi \in \Int (n)} \cf_{(i_1, \ldots , i_n); \pi} ( \eta_{\mu} )
\end{equation}
(where (\ref{eqn:2.141}), (\ref{eqn:2.142}) use the notations for 
generalized coefficients from Definition \ref{def:2.11}.2, and 
$\Int (n)$ is the set of interval-partitions from 
Remark \ref{rem:2.9}). Moreover, a similar summation formula 
can be used in order to express the Boolean 
cumulants of $\mu$ in terms of its free cumulants; it says that 
for every $n \geq 1$ and $1 \leq i_1, \ldots , i_n \leq k$ one has
\begin{equation}  \label{eqn:2.143}
\cf_{(i_1, \ldots , i_n)} ( \eta_{\mu} ) 
= \sum_{ \begin{array}{c}
{\scriptstyle  \pi \in NC(n),}  \\
{\scriptstyle \pi \leqleq 1_n}
\end{array}  } \ \cf_{(i_1, \ldots , i_n); \pi} ( R_{\mu} ).
\end{equation}
(For a more detailed discussion of the relation between $R_{\mu}$
and $\eta_{\mu}$ see Section 3 of \cite{BN08}, where Equation
(\ref{eqn:2.143}) appears in Proposition 3.9.)
\end{remark}

$\ $

$\ $

$\ $

$\ $

\begin{center}
{\bf\large 3. The approach to $\boxright$ via R-transforms}
\end{center}
\setcounter{section}{3}
\setcounter{equation}{0}
\setcounter{theorem}{0}

\noindent
The goal of this section is to derive explicit combinatorial 
formulas for the free and Boolean cumulants of $\mu \boxright \nu$, 
and then use them in order to obtain the moment formula announced 
in Theorem \ref{thm:1.3}.

\begin{remark}   \label{rem:3.1}
Let $\mu, \nu$ be distributions in $\dalg (k)$. Consider the 
subordination distribution $\mu \boxright \nu$, and recall that 
its $R$-transform satisfies the equation
\begin{equation}  \label{eqn:3.11}
R_{\mu \boxright \nu}  \cdot ( 1+M_{\nu} ) =
R_{\mu} \Bigl( \, z_1 (1+M_{\nu}), \ldots , z_k (1+M_{\nu} \, \Bigr).
\end{equation}
If we denote for convenience 
\[
\cf_{ (i_1, \ldots , i_n) } ( R_{\mu} ) =: 
\alpha_{ (i_1, \ldots , i_n) }, \ \ \forall \, n \geq 1, \
1 \leq i_1, \ldots , i_n \leq k,
\]
then the series on the right-hand side of (\ref{eqn:3.11}) is written
more precisely as
\begin{equation}  \label{eqn:3.12}
\sum_{m=1}^{\infty} \ \sum_{j_1, \ldots , j_m =1}^k \,
\alpha_{ (j_1, \ldots , j_m) } 
z_{j_1} (1+ M_{\nu}) \cdots z_{j_m} (1+ M_{\nu}).
\end{equation}

Let us fix an $n \geq 1$ and some indices 
$1 \leq i_1, \ldots , i_n \leq k$, and let us look at the coefficient
of $z_{i_1} \cdots z_{i_n}$ in the infinite sum from (\ref{eqn:3.12}).
Clearly, a term $\alpha_{ (j_1, \ldots , j_m) } z_{j_1}$
$(1+ M_{\nu}) \cdots z_{j_m} (1+ M_{\nu})$
contributes to this coefficient if and only if $m \leq n$ and there
exist $1 = s(1) < s(2) < \cdots < s(m) \leq n$ such that 
\begin{equation}  \label{eqn:3.13}
j_1 = i_{s(1)}, \ j_2 = i_{s(2)}, \ldots , j_m = i_{s(m)}.
\end{equation}
In the case when (\ref{eqn:3.13}) holds let us denote 
$\{ s(1), s(2), \ldots , s(m) \} =: S$, and let us refer to the 
intervals of integers
\[
( s(1), s(2) ) \cap \bZ, \ldots ,
( s(m-1), s(m) ) \cap \bZ, \  ( s(m), n ] \cap \bZ
\]
by calling them the {\em gaps of $S$}; with this notation the 
contribution of 
$\alpha_{ (j_1, \ldots , j_m) } z_{j_1} (1+ M_{\nu})$
$\cdots z_{j_m} (1+ M_{\nu})$
to the coefficient of $z_{i_1} \cdots z_{i_n}$ in (\ref{eqn:3.12}) 
is written as
\[
\alpha_{ (i_1, \ldots , i_n) | S} \cdot
\prod_{ \begin{array}{c}
{\scriptstyle G = \{ p, \ldots , q \} } \\
{\scriptstyle gap \ of \ S}
\end{array}  } \nu ( X_{i_p} \cdots X_{i_q} )
\]
(we make the convention that if $G$ is an empty gap of $S$ then 
the corresponding product $\nu ( X_{i_p} \cdots X_{i_q} )$ is taken
to be equal to 1). Since the set $S$ appearing above can be any 
subset of $\{ 1, \ldots , n \}$ which contains 1, we come to the 
conclusion that 
\begin{equation}  \label{eqn:3.14}
\cf_{ (i_1, \ldots , i_n) } \Bigl( \,
R_{\mu} \bigl( \, z_1 (1+M_{\nu}), \ldots , z_k (1+M_{\nu}) 
\, \bigr) \, \Bigr)
\end{equation}
\[
= \sum_{\begin{array}{c}
{\scriptstyle S \subseteq \{1, \ldots , n \} } \\
{\scriptstyle such \ that \ S \ni 1 }
\end{array}  } \, \Bigl( \alpha_{ (i_1, \ldots , i_n) | S} \cdot
\prod_{ \begin{array}{c}
{\scriptstyle G = \{ p, \ldots , q \} } \\
{\scriptstyle gap \ of \ S}
\end{array}  } \nu ( X_{i_p} \cdots X_{i_q} ) \Bigr) .
\]

By equating coefficients in the series on the two sides of 
(\ref{eqn:3.11}) and by employing (\ref{eqn:3.14}) one obtains 
explicit formulas for the coefficients of $R_{\mu \boxright \nu}$, 
as shown in the next lemma and proposition.
\end{remark}

\begin{lemma}   \label{lemma:3.2}
Consider the same notations as in Remark \ref{rem:3.1}. For every
$n \geq 1$ and $1 \leq i_1, \ldots , i_n \leq k$ one has that 
\begin{equation}  \label{eqn:3.21}
\cf_{ (i_1, \ldots , i_n) } \bigl( R_{\mu \boxright \nu} \bigr) 
= \sum_{\begin{array}{c}
{\scriptstyle S \subseteq \{1, \ldots , n \} } \\
{\scriptstyle such \ that \ S \ni 1,n }
\end{array}  } \, \Bigl( \alpha_{ (i_1, \ldots , i_n) | S} \cdot
\prod_{ \begin{array}{c}
{\scriptstyle G = \{ p, \ldots , q \} } \\
{\scriptstyle gap \ of \ S}
\end{array}  } \nu ( X_{i_p} \cdots X_{i_q} ) \Bigr) .
\end{equation}
\end{lemma}

\begin{proof} 
We will prove the required formula (\ref{eqn:3.21}) by 
induction on $n$.

For $n=1$, (\ref{eqn:3.21}) states that 
$\cf_{(i_1)} (R_{\mu \boxright \nu} ) = \alpha_{i_1}$, 
$\forall \, 1 \leq i_1 \leq k$; this is indeed true, as one 
sees by equating the coefficients of $z_{i_1}$ 
on the two sides of (\ref{eqn:3.11}).

Induction step: we fix an integer $n \geq 2$, we assume that 
(\ref{eqn:3.21}) holds for $1,2, \ldots , n-1$ and we prove 
that it also holds for $n$. So let $i_1, \ldots , i_n$
be some indices in $\{ 1, \ldots , k \}$. The coefficient of 
$z_{i_1} \cdots z_{i_n}$ in 
$R_{\mu \boxright \nu} \cdot (1+ M_{\nu} )$ is equal to:
\begin{equation}  \label{eqn:3.22}
\cf_{ (i_1, \ldots , i_n) } ( R_{\mu \boxright \nu} ) 
+ \sum_{m=1}^{n-1} 
\cf_{ (i_1, \ldots , i_m) } ( R_{\mu \boxright \nu} ) \cdot
\nu ( X_{i_{m+1}} \cdots X_{i_n} ).
\end{equation}
For every $1 \leq m \leq n-1$ the induction hypothesis gives 
us that
\[
\cf_{ (i_1, \ldots , i_m) } ( R_{\mu \boxright \nu} ) \cdot
\nu ( X_{i_{m+1}} \cdots X_{i_n} ) =
\]
\[
\sum_{ \begin{array}{c}
{\scriptstyle S \subseteq \{ 1, \ldots , m \} } \\
{\scriptstyle such \ that \ S \ni 1,m}
\end{array} } \ \alpha_{ (i_1, \ldots , i_m) | S} \cdot
\Bigl( \, \prod_{ \begin{array}{c}
{\scriptstyle G = \{ p, \ldots , q \} } \\
{\scriptstyle gap \ of \ S}
\end{array}  } \nu ( X_{i_p} \cdots X_{i_q} ) \, \Bigr) 
\cdot \nu ( X_{i_{m+1}} \cdots X_{i_n} ) .
\]
In the latter expression
the separate factor $\nu ( X_{i_{m+1}} \cdots X_{i_n} )$ can be 
incorporated into the product over the gaps of $S$, via the simple
trick of treating $S$ as a subset of $\{ 1, \ldots , n \}$
rather than a subset of $\{ 1, \ldots , m \}$. (Indeed, in this 
way $S$ gets the additional gap $\{ m+1, \ldots , n \}$, with 
corresponding factor $\nu ( X_{i_{m+1}} \cdots X_{i_n} )$.)
When this is done and when the resulting formula for 
$\cf_{ (i_1, \ldots , i_m) } ( R_{\mu \boxright \nu} ) \cdot
\nu ( X_{i_{m+1}} \cdots X_{i_n} )$
is substituted in (\ref{eqn:3.22}), we find that:
\begin{equation}  \label{eqn:3.23}
\cf_{ (i_1, \ldots , i_n) } \Bigl( \, R_{\mu \boxright \nu} \cdot 
(1+ M_{\nu}) \, \Bigr) = 
\cf_{ (i_1, \ldots , i_n) } ( R_{\mu \boxright \nu} ) 
\end{equation}
\[
+ \sum_{ \begin{array}{c}
{\scriptstyle S \subseteq \{ 1, \ldots , n \} } \\
{\scriptstyle such \ that \ S \ni 1 \ and \ S \not\ni n}
\end{array} } \ \alpha_{ (i_1, \ldots , i_n) | S} \cdot
\Bigl( \, \prod_{ \begin{array}{c}
{\scriptstyle G = \{ p, \ldots , q \} } \\
{\scriptstyle gap \ of \ S}
\end{array}  } \nu ( X_{i_p} \cdots X_{i_q} ) \Bigr) .
\]
Finally, we equate the right-hand sides of Equations (\ref{eqn:3.23})
and (\ref{eqn:3.14}), and the required formula for 
$\cf_{ (i_1, \ldots , i_n) } \bigl( \, R_{\mu \boxright \nu}\, \bigr)$
follows.
\end{proof}

\begin{proposition}   \label{prop:3.3}
Let $\mu, \nu$ be distributions in $\dalg (k)$. For every 
$n \geq 1$ and $1 \leq i_1, \ldots , i_n \leq k$ one has
\begin{equation}    \label{eqn:3.31}
\cf_{ (i_1, \ldots , i_n) } ( R_{\mu \boxright \nu} ) = 
\sum_{  \begin{array}{c} 
{\scriptstyle \pi \in NC(n),}  \\
{\scriptstyle \pi \leqleq 1_n} 
\end{array} } \ \cf_{ (i_1, \ldots , i_n); \pi ; \oo_{\pi} }
( R_{\mu}, R_{\nu} ),
\end{equation}
where the inner/outer colouring $\oo_{\pi}$ is as in Notation
\ref{def:2.4}.5, and the generalized coefficient 
$\cf_{ (i_1, \ldots , i_n); \pi ; \oo_{\pi} } ( R_{\mu}, R_{\nu} )$
is as in Definition \ref{def:2.11}.3.
\end{proposition}

\begin{proof} We will use the various notations introduced in 
Remark \ref{rem:3.1} and Lemma \ref{lemma:3.2} above. 

Let us pick a subset $S \subseteq \{ 1, \ldots , n \}$ such that 
$S \ni 1,n$, and let us prove that
\begin{equation}    \label{eqn:3.32}
\alpha_{ (i_1, \ldots , i_n) | S} \cdot
\Bigl( \, \prod_{ \begin{array}{c}
{\scriptstyle G = \{ p, \ldots , q \} } \\
{\scriptstyle gap \ of \ S}
\end{array}  } \nu ( X_{i_p} \cdots X_{i_q} ) \Bigr) 
= \sum_{  \begin{array}{c} 
{\scriptstyle \pi \in NC(n)} \ such  \\
{\scriptstyle that  \ S \in \pi} 
\end{array} } \ \cf_{ (i_1, \ldots , i_n); \pi ; \opi }
( R_{\mu}, R_{\nu} ).
\end{equation}
In order to verify (\ref{eqn:3.32}),
let us write explicitly $S = \{ s(1), s(2), \ldots , s(m) \}$
with $1 = s(1) < s(2) < \cdots < s(m) =n$; then the gaps of $S$ 
are listed as $G_1, \ldots , G_{m-1}$, with
\[
G_j = \{ p_j, \ldots , q_j \} = (s(j), s(j+1)) \cap \bZ
\mbox{ for } 1 \leq j \leq m-1,
\]
and the left-hand side of (\ref{eqn:3.32}) becomes
\begin{equation}    \label{eqn:3.33}
\alpha_{ ( i_1, \ldots , i_n ) | S } \cdot
\prod_{j=1}^{m-1} 
\nu ( X_{i_{p_j}} \cdots X_{i_{q_j}} ) 
\end{equation}
(with the same convention as used above, that 
``$\nu ( X_{i_{p_j}} \cdots X_{i_{q_j}} )$'' is to be read as 1 in 
the case when $G_j = \emptyset$). Now in (\ref{eqn:3.33}) let us 
use the free moment-cumulant formula (\ref{eqn:2.141}) to express 
the moments $\nu ( X_{i_{p_j}} \cdots X_{i_{q_j}} )$ in terms of the 
coefficients of $R_{\nu}$; we get
\[
\alpha_{ (i_1, \ldots , i_n) \mid S } \cdot
\prod_{j=1}^{m-1}  \Bigl( \, \sum_{ \pi_j \in NC(|G_j|) } \,
\cf_{ (i_{p_j}, \ldots , i_{q_j}); \pi_j } (R_{\nu}) \, \Bigr)
\]
\begin{equation}    \label{eqn:3.34}
= \sum_{ \begin{array}{c}
{\scriptstyle \pi_1 \in NC(|G_1|), \ldots }    \\
{\scriptstyle \ldots , \pi_{m-1} \in NC(|G_{m-1}|) } 
\end{array} } \ \Bigl( 
\cf_{ (i_1, \ldots , i_n) \mid S } (R_{\mu}) \cdot
\prod_{j=1}^{m-1} 
\cf_{ (i_1, \ldots , i_n) \mid G_j ); \pi_j} (R_{\nu})  \Bigr).
\end{equation}
But a family of non-crossing partitions 
$\pi_1 \in NC(|G_1|), \ldots , \pi_{m-1} \in NC(|G_{m-1}|)$
is naturally assembled, together with $S$, into one non-crossing
partition $\pi \in NC(n)$; and all partitions $\pi \in NC(n)$ 
such that $S \in \pi$ are obtained in this way, without 
repetitions. Moreover, when $\pi_1, \ldots , \pi_{m-1}$ and $S$ 
are assembled together into $\pi$, it is clear that the 
big product from (\ref{eqn:3.34}) becomes just
$\cf_{ (i_1, \ldots , i_n); \pi; \oo_{\pi} } (R_{\mu}, R_{\nu})$.
Hence the substitution 
$( \pi_1, \ldots , \pi_{m-1} ) \leftrightarrow \pi$ leads to the
right-hand side of (\ref{eqn:3.32}), and this completes the proof 
that (\ref{eqn:3.32}) holds.

Finally, we sum over $S$ on both sides of (\ref{eqn:3.32}), with 
$S$ running in the collection of all subsets of $\{ 1, \ldots ,n \}$
which contain $1$ and $n$. The sum on the left-hand side gives
$\cf_{ (i_1, \ldots , i_n) } ( R_{\mu \boxright \nu} )$ by
Lemma \ref{lemma:3.2}, while the sum on the right-hand side  
takes us precisely to the right-hand side of (\ref{eqn:3.31}), 
as we wanted.
\end{proof}

It will come in handy to also have an extended version of the 
formula found in Proposition \ref{prop:3.3}, which covers the
generalized coefficients 
``$\cf_{ (i_1, \ldots , i_n); \rho}$'' of the $R$-transform of 
$\mu \boxright \nu$. This is presented in Lemma \ref{lemma:3.6},
and uses the following extension for the concept
of inner/outer colouring of a non-crossing partition.

\begin{notation}   \label{def:3.4}
Let $n$ be a positive integer and let $\pi, \rho$ be partitions
in $NC(n)$ such that $\pi \leqleq \rho$. We denote by 
$\oo_{\pi , \rho}$ the colouring of $\pi$ defined by
\begin{equation}   \label{eqn:3.41}
\oo_{\pi ; \rho} (V) = 
\left \{   \begin{array}{cl}
1,  & \mbox{if $V$ is $\rho$-special}   \\
2,  & \mbox{if $V$ is not $\rho$-special,} 
\end{array}   \right.   \ \ \ V \in \pi ,
\end{equation}
where the concept of ``being $\rho$-special'' for a block of $\pi$
is as in Definition \ref{def:2.7}.
\end{notation}

\begin{remark}    \label{rem:3.5}
Let $\pi$ be a partition in $NC(n)$ and let $\rho$ be the unique 
interval-partition with the property that $\rho \geqgeq \pi$. Then 
the colouring $\oo_{\pi, \rho}$ defined above is just the usual 
inner/outer colouring $\oo_{\pi}$ -- indeed, in this case a block
$V$ of $\pi$ is $\rho$-special if and only if it is outer.
\end{remark}

\begin{lemma}   \label{lemma:3.6}
Let $\mu, \nu$ be distributions in $\dalg (k)$. For every 
$n \geq 1$, $\rho \in NC(n)$ and $1 \leq i_1, \ldots , i_n \leq k$ 
one has
\begin{equation}    \label{eqn:3.61}
\cf_{ (i_1, \ldots , i_n); \rho } ( R_{\mu \boxright \nu} ) = 
\sum_{  \begin{array}{c} 
{\scriptstyle \pi \in NC(n),}  \\
{\scriptstyle \pi \leqleq \rho} 
\end{array} } \ \cf_{ (i_1, \ldots , i_n); \pi ; \oo_{\pi, \rho} }
( R_{\mu}, R_{\nu} ).
\end{equation}
\end{lemma}

\begin{proof}
Let us write explicitly $\rho = \{ W_1, \ldots , W_q \}$. Then
\[
\cf_{ (i_1, \ldots , i_n); \rho } ( R_{\mu \boxright \nu} ) = 
\prod_{j=1}^q 
\cf_{ (i_1, \ldots , i_n) | W_j } ( R_{\mu \boxright \nu} ) 
\]
\[
= \prod_{j=1}^q \Bigl( \, \sum_{ \begin{array}{c}
{\scriptstyle \pi_j \in NC( |W_j| ), } \\
{\scriptstyle \pi_j \leqleq 1_{|W_j|} }
\end{array}  } \,
\cf_{ ((i_1, \ldots , i_n) | W_j); \pi_j; \oo_{\pi_j} } 
( R_{\mu} , R_{\nu} ) \, \Bigr) 
\]
\begin{equation}    \label{eqn:3.62}
= \sum_{\begin{array}{c}
{\scriptstyle  \pi_1 \in NC( |W_1| ), 
               \pi_1 \leqleq 1_{|W_1|}, \ldots }   \\
{\scriptstyle  \ldots , \pi_q \in NC( |W_q| ), 
                        \pi_q \leqleq 1_{|W_q|} }   
\end{array} } \
\Bigl( \ \prod_{j=1}^q 
\cf_{ ((i_1, \ldots , i_n) | W_j); \pi_j; \oo_{\pi_j} } 
( R_{\mu} , R_{\nu} ) \ \Bigr) .
\end{equation}

Now let us consider the bijection (\ref{eqn:2.51}) from
Remark \ref{rem:2.5}. It is immediate that if
$\pi \leftrightarrow ( \pi_1, \ldots , \pi_q )$ via this
bijection, then
\[
\prod_{j=1}^q 
\cf_{ ((i_1, \ldots , i_n) | W_j); \pi_j; \oo_{\pi_j} } 
( R_{\mu} , R_{\nu} ) = 
\cf_{ (i_1, \ldots , i_n); \pi; \oo_{\pi, \rho} } 
( R_{\mu} , R_{\nu} ).
\]
Thus when in (\ref{eqn:3.62}) we perform the change of variable
given by the bijection from (\ref{eqn:2.51}), we arrive precisely
to the right-hand side of (\ref{eqn:3.61}), as required.
\end{proof}

On our way towards the formula for moments stated in Theorem 
\ref{thm:1.3} we next put into evidence an explicit formula for the 
Boolean cumulants of $\mu \boxright \nu$.

\begin{proposition}   \label{prop:3.7}
Let $\mu, \nu$ be distributions in $\dalg (k)$. For every 
$n \geq 1$ and $1 \leq i_1, \ldots , i_n \leq k$ one has
\begin{equation}    \label{eqn:3.71}
\cf_{ (i_1, \ldots , i_n) } ( \eta_{\mu \boxright \nu} ) = 
\sum_{  \begin{array}{c} 
{\scriptstyle \pi \in NC(n), \ \pi \leqleq 1_n}  \\
{\scriptstyle with \ outer \ block \ V_o} 
\end{array} } \ \
\sum_{  \begin{array}{c} 
{\scriptstyle c: \pi \to \{ 1, 2 \} \ such } \\
{\scriptstyle that \ c(V_o) =1}
\end{array} } \ \
\cf_{ (i_1, \ldots , i_n); \pi ; c }
( R_{\mu}, R_{\nu} ).
\end{equation}
Moreover, for every $\pi \in NC(n)$, $\pi \leqleq 1_n$ with outer
block $V_o$, one has:
\begin{equation}    \label{eqn:3.72}
\sum_{  \begin{array}{c} 
{\scriptstyle c: \pi \to \{ 1, 2 \} \ such } \\
{\scriptstyle that \ c(V_o) =1}
\end{array} } \ \
\cf_{ (i_1, \ldots , i_n); \pi ; c }
( R_{\mu}, R_{\nu} )
= \cf_{ (i_1, \ldots , i_n); \pi ; \opi }
( R_{\mu}, R_{\mu} + R_{\nu} ).
\end{equation}
Hence Equation (\ref{eqn:3.71}) can also be written in the form
\begin{equation}    \label{eqn:3.73}
\cf_{ (i_1, \ldots , i_n) } ( \eta_{\mu \boxright \nu} ) = 
\sum_{  \begin{array}{c} 
{\scriptstyle \pi \in NC(n),}  \\
{\scriptstyle \pi \leqleq 1_n} 
\end{array} } \ \cf_{ (i_1, \ldots , i_n); \pi ; \opi }
( R_{\mu}, R_{\mu} + R_{\nu} ).
\end{equation}
\end{proposition}

\begin{proof} It is immediate that the left-hand side of 
(\ref{eqn:3.72}) is merely the expansion as a sum for the product
which defines $\cf_{ (i_1, \ldots , i_n); \pi ; \opi }
( R_{\mu}, R_{\mu} + R_{\nu} )$. Hence the only non-trivial point 
in this proof is to verify that (\ref{eqn:3.71}) holds.

By using how $\cf_{ (i_1, \ldots , i_n) } 
\bigl( \eta_{\mu \boxright \nu} \bigr)$ is written in terms of 
the coefficients of $R_{\mu \boxright \nu}$ (cf. Equation 
(\ref{eqn:2.143}) in Remark \ref{rem:2.14}), then 
by invoking Lemma \ref{lemma:3.6} and by performing an obvious 
change in the order of summation we get that
\[
\cf_{ (i_1, \ldots , i_n) } ( \eta_{\mu \boxright \nu} ) 
= \sum_{ \begin{array}{c}
{\scriptstyle \rho \in NC(n),}   \\
{\scriptstyle \rho \leqleq 1_n}
\end{array} } \,
\cf_{ (i_1, \ldots , i_n); \rho } ( R_{\mu \boxright \nu} )
\]
\[
= \sum_{ \begin{array}{c}
{\scriptstyle \rho \in NC(n),}   \\
{\scriptstyle \rho \leqleq 1_n}
\end{array} } \,
\Bigl( \,
\sum_{  \begin{array}{c} 
{\scriptstyle \pi \in NC(n),}  \\
{\scriptstyle \pi \leqleq \rho} 
\end{array} } \ \cf_{ (i_1, \ldots , i_n); \pi ; \oo_{\pi, \rho} }
( R_{\mu}, R_{\nu} ) \, \Bigr)
\]
\[
= \sum_{ \begin{array}{c}
{\scriptstyle \pi \in NC(n),}   \\
{\scriptstyle \pi \leqleq 1_n}
\end{array} } \,
\Bigl( \, \sum_{  \begin{array}{c} 
{\scriptstyle \rho \in NC (n) \ such}  \\
{\scriptstyle that \ \pi \leqleq \rho \leqleq 1_n} 
\end{array} } \ \cf_{ (i_1, \ldots , i_n); \pi ; \oo_{\pi, \rho} }
( R_{\mu}, R_{\nu} ) \, \Bigr) .
\]
In order to conclude the proof we are left to show that for every 
partition $\pi \in NC(n)$ with $\pi \leqleq 1_n$ and with outer 
block denoted $V_0$ one has
\begin{equation}    \label{eqn:3.74}
\sum_{  \begin{array}{c} 
{\scriptstyle \rho \in NC (n) \ such}  \\
{\scriptstyle that \ \pi \leqleq \rho \leqleq 1_n} 
\end{array} } \ \cf_{ (i_1, \ldots , i_n); \pi ; \oo_{\pi, \rho} }
( R_{\mu}, R_{\nu} ) 
= \sum_{  \begin{array}{c} 
{\scriptstyle c: \pi \to \{ 1,2 \} }  \\
{\scriptstyle such \ that \ c(V_o)=1}
\end{array} } \ \cf_{ (i_1, \ldots , i_n); \pi ; c }
( R_{\mu}, R_{\nu} ).
\end{equation}
And indeed, recall from Proposition \ref{prop:2.8} 
that we have a bijection
\[
\begin{array}{rcl}
\{ \rho \in NC(n) \mid \pi \leqleq \rho \leqleq 1_n \} &
       \to & \{ \fV \subseteq \pi \mid \fV \ni V_0 \}       \\
\rho & \mapsto & \{ V \in \pi \mid V \mbox{ is $\rho$-special} \} .
\end{array}
\]
When comparing this bijection against the formula which defined 
$\oo_{\pi , \rho}$ in Notation \ref{def:3.4}, it is immediate that 
the map $\rho \mapsto \oo_{\pi , \rho}$ is itself a bijection from
$\{ \rho \in NC(n) \mid \pi \leqleq \rho \leqleq 1_n \}$ onto
the set of colourings $\{ c: \pi \to \{ 1,2 \} \mid c(V_0) =1 \}$,
and (\ref{eqn:3.74}) immediately follows.
\end{proof}

\begin{remark}   \label{rem:3.8}

$1^o$ When considered together, Equations (\ref{eqn:3.73}) and 
(\ref{eqn:3.31}) give that 
\begin{equation}    \label{eqn:3.81}
\eta_{\mu \boxright \nu} = R_{\mu \boxright ( \mu \boxplus \nu )};
\end{equation}
the latter formula is in turn telling us that 
\begin{equation}    \label{eqn:3.82}
\bB ( \mu \boxright \nu ) = \mu \boxright ( \mu \boxplus \nu ),
\end{equation}
where $\bB$ is the Boolean Bercovici-Pata bijection on $\dalg (k)$.
Equation (\ref{eqn:3.82}) is a special case of Proposition 
\ref{prop:1.10} of the introduction; but actually the general case 
of Proposition \ref{prop:1.10} easily follows from here, as 
explained in the proof of Proposition \ref{prop:5.1} below.

$2^o$ In the same way as the statement of Proposition \ref{prop:3.3}
was extended to the one of Lemma \ref{lemma:3.6}, the formula 
found in Proposition \ref{prop:3.7} can be extended to
\begin{equation}    \label{eqn:3.83}
\cf_{ (i_1, \ldots , i_n); \rho } ( \eta_{\mu \boxright \nu} ) = 
\sum_{  \begin{array}{c} 
{\scriptstyle \pi \in NC(n),}  \\
{\scriptstyle \pi \leqleq \rho} 
\end{array} } \ \cf_{ (i_1, \ldots , i_n); \pi ; \oo_{\pi, \rho} }
( R_{\mu}, R_{\mu} + R_{\nu} ),
\end{equation}
holding for every $n \geq 1$, $\rho \in NC(n)$, and 
$1 \leq i_1, \ldots , i_n \leq k$.
Equation (\ref{eqn:3.83}) can be obtained from (\ref{eqn:3.73}) by 
an argument similar to the one used in the proof of Lemma 
\ref{lemma:3.6}; but in fact we don't need to repeat that argument,
we can simply infer (\ref{eqn:3.83}) by using Lemma \ref{lemma:3.6} 
itself, in conjunction to Equation (\ref{eqn:3.81})
from the first part of the present remark.
\end{remark}

\vspace{6pt}

It is now easy to obtain the moment formula stated in Theorem 
\ref{thm:1.3} of the introduction.

\begin{proposition}   \label{prop:3.9}
Let $\mu, \nu$ be distributions in $\dalg (k)$. For every 
$n \geq 1$ and $1 \leq i_1, \ldots , i_n \leq k$ one has
\begin{equation}    \label{eqn:3.91}
( \mu \boxright \nu ) (X_{i_1} \cdots X_{i_n}) =
\sum_{ \pi \in NC(n) } \cf_{ (i_1, \ldots , i_n); \pi ; \opi }
( R_{\mu}, R_{\mu} + R_{\nu} ).
\end{equation}
\end{proposition}

\begin{proof} By using the Boolean moment-cumulant formula
(Equation (\ref{eqn:2.142}) in Remark \ref{rem:2.14}),
then by invoking Remark \ref{rem:3.8}.2 and by performing an 
obvious change in the order of summation we get that
\[
( \mu \boxright \nu ) (X_{i_1} \cdots X_{i_n}) =
\sum_{ \rho \in \Int (n) } \cf_{ (i_1, \ldots , i_n); \rho }
( \eta_{\mu \boxright \nu} )
\]
\[
= \sum_{ \rho \in \Int (n) }  \Bigl( \,
\sum_{  \begin{array}{c} 
{\scriptstyle \pi \in NC(n),}  \\
{\scriptstyle \pi \leqleq \rho} 
\end{array} } \ \cf_{ (i_1, \ldots , i_n); \pi ; \oo_{\pi, \rho} }
( R_{\mu}, R_{\mu} + R_{\nu} ) \, \Bigr)
\]
\begin{equation}   \label{eqn:3.92}
= \sum_{ \pi \in NC(n) }  \Bigl( \,
\sum_{  \begin{array}{c} 
{\scriptstyle \rho \in \Int (n),}  \\
{\scriptstyle \rho \geqgeq \pi} 
\end{array} } \ \cf_{ (i_1, \ldots , i_n); \pi ; \oo_{\pi, \rho} }
( R_{\mu}, R_{\mu} + R_{\nu} ) \, \Bigr) .
\end{equation}
But for every $\pi \in NC(n)$ there exists a unique partition 
$\rho \in \Int (n)$ such that $\rho \geqgeq \pi$, and for this $\rho$
we have $\oo_{\pi , \rho} = \oo_{\pi}$ (as observed in Remark 
\ref{rem:3.5}). Thus the sum over $\rho$ in (\ref{eqn:3.92})
consists of just one term, 
$\cf_{ (i_1, \ldots , i_n); \pi ; \oo_{\pi} }
( R_{\mu}, R_{\mu} + R_{\nu} )$, and (\ref{eqn:3.91}) follows.
\end{proof}

\begin{remark}    \label{rem:3.10}

$1^o$ A summation of the same type as in Equation (\ref{eqn:3.91}), 
which uses coefficients from two series and distinguishes between the 
inner and outer blocks of $\pi \in NC(n)$, has previously appeared in 
the theory of $c$-free convolution -- see e.g. the third 
displayed equation on p. 366 of \cite{BLS96}. This connection is 
not pursued in the present paper, but $c$-free convolution is 
heavily used in \cite{A08} (which relates to the present paper in 
the way explained in Remark \ref{rem:1.14} of the introduction).

$2^o$ In the proof of Theorem \ref{thm:4.4} of the next section 
we will also need the equivalent form of Equation (\ref{eqn:3.91}) 
where, for every $\pi \in NC(n)$, the product defining 
$\cf_{ (i_1, \ldots , i_n); \pi; \oo_{\pi} } 
(R_{\mu}, R_{\mu} + R_{\nu})$
is expanded into a sum. It is immediate (left as exercise to the 
reader) to check that the formula for the moments of 
$\mu \boxright \nu$ will then look as follows:
\begin{equation}   \label{eqn:3.101}
( \mu \boxright \nu ) ( X_{i_1} \cdots X_{i_n} )
= \sum_{ ( \pi , c ) }
\cf_{ (i_1, \ldots , i_n); \pi ; c } ( R_{\mu} , R_{\nu} ) ,
\end{equation}
where the index set for the sum on the right-hand side of 
(\ref{eqn:3.101}) is
\[
\left\{  ( \pi , c ) \
\begin{array}{cl}
\vline & \pi \in NC(n), \ c: \pi \to \{ 1,2 \}, \mbox{ such that} \\
\vline & \mbox{$c(V)=1$ for every outer block $V$ of $\pi$}
\end{array} \right\}  .
\]
\end{remark}

\begin{remark}    \label{rem:3.11}
Let $\mu$ and $( \mu_N )_{N \geq 1}$ be in $\dalg (k)$. If 
\begin{equation}    \label{eqn:3.111}
\lim_{N \to \infty}  \mu_N (X_{i_1} \cdots X_{i_n}) = 
\mu (X_{i_1} \cdots X_{i_n}) , \ \ \ \forall \, 
n \geq 1, \ \forall \, 1 \leq i_1, \ldots , i_n \leq k,
\end{equation}
then one says that the sequence $( \mu_N )_{N \geq 1}$ 
{\em converges in distribution} to $\mu$ (denoted simply as
$\mu_N \to \mu$). Due to the moment-cumulant 
formulas from Remark \ref{rem:2.14}, this is equivalent to convergence
in coefficients for the $R$-transforms $R_{\mu_N}$ to $R_{\mu}$, or 
for the $\eta$-series $\eta_{\mu_N}$ to $\eta_{\mu}$.

Now, from the fact that one has polynomial expressions giving the 
moments of $\mu \boxright \nu$ in terms of the free cumulants of 
$\mu$ and of $\nu$ it is immediate that the operation $\boxright$
is well-behaved under taking limits in distribution in 
$\dalg (k)$. That is, if $\mu, \nu$, 
$( \mu_N )_{N=1}^{\infty}$ and $( \nu_N )_{N=1}^{\infty}$ are 
distributions in $\dalg (k)$ such that $\mu_N \to \mu$ and 
$\nu_N \to \nu$, then it follows that
$\mu_N \boxright \nu_N \to \mu \boxright \nu$. The same conclusion
could have been of course derived directly from Proposition 
\ref{prop:3.3}, or from Proposition \ref{prop:3.7}.
\end{remark}

$\ $

$\ $

\begin{center}
{\bf\large 4. The approach to $\boxright$ via operator models}
\end{center}
\setcounter{section}{4}
\setcounter{equation}{0}
\setcounter{theorem}{0}

This section puts into evidence a full Fock space model for 
$\mu \boxright \nu$, then uses this model in order to obtain 
Theorem \ref{thm:1.4} stated in the introduction of the paper.

The full Fock space model is given in Theorem \ref{thm:4.4}, and
is just a variation of the ``standard'' full Fock space model for 
the $R$-transform (as presented for instance in Lecture 21 of 
\cite{NS06}). In order to avoid tedious notations involving formal 
operators on the full Fock spoace, we will only consider this model 
in the special case when the $R$-transforms 
$R_{\mu}$ and $R_{\nu}$ are polynomials. A more general statement 
could be obtained from this special case by doing approximations 
in distribution (a very similar procedure to how Theorem 21.4 is 
extended to Theorem 21.7 in Lecture 21 of \cite{NS06}). However, 
for the situation at hand it is actually more convenient to 
incorporate the necessary approximations in distribution directly 
into the proof of Theorem \ref{thm:4.9} below, where the full Fock 
space model is upgraded to the more general framework of Theorem 
\ref{thm:1.4}.

\begin{notation}   \label{def:4.1}
Let $\cF$ be the full Fock space over $\bC^{2k}$,
\[
\cF := \bC \oplus \bC^{2k} 
\oplus \bigl( \, \bC^{2k} \, \bigr)^{\otimes 2} \oplus \cdots
\oplus \bigl( \, \bC^{2k} \, \bigr)^{\otimes n} \oplus \cdots 
\]
The vector 
$1 \oplus 0 \oplus 0 \oplus \cdots \oplus 0 \oplus \cdots$
is called the {\em vacuum-vector} of $\cF$ and is denoted by 
$\Omega$. We will let $\pvac \in B( \cF )$ denote the orthogonal 
projection onto the 1-dimensional space $\bC \Omega \subseteq \cF$.
The vector-state 
$T \mapsto \langle T \Omega \, , \, \Omega \rangle$
defined by $\Omega$ on $B( \cF )$ will be referred to as 
{\em vacuum-state}.

We fix an orthonormal basis for $\bC^{2k}$, which we denote as 
$e_1 ', \ldots , e_k ', e_1 '', \ldots , e_k ''$. This leads to
a natural choice of orthonormal basis for $\cF$,
\begin{equation}   \label{eqn:4.11}
\{ \Omega \} \cup \Bigl\{ \, \xi_1 \otimes \cdots \otimes \xi_n 
\mid n \geq 1, \ \xi_1, \ldots , \xi_n \in \{ 
e_1 ', \ldots , e_k ', e_1 '', \ldots , e_k \} \, \Bigr\} .
\end{equation}
For every $1 \leq i \leq k$ the left creation operators with
$e_i'$ and $e_i''$ will be denoted by $L_i'$ and $L_i''$,
respectively. So $L_i' \in B( \cF )$ is the isometry which acts 
on the orthonormal basis (\ref{eqn:4.11}) by
\[
L_i' ( \Omega ) = e_i', \ \ \
L_i' ( \xi_1 \otimes \cdots \otimes \xi_n ) = 
e_i' \otimes \xi_1 \otimes \cdots \otimes \xi_n,
\]
and similar formulas hold for $L_i''$. Moreover, we will denote
by $\fM '$ and $\fM ''$ the sets of operators in $B ( \cF )$ 
defined by:
\begin{equation}   \label{eqn:4.12}
\left\{   \begin{array}{ccl}
\fM ' & := & \{ L_{i_1}' \cdots L_{i_n}'  \mid
n \geq 1, \ 1 \leq i_1, \ldots , i_n \leq k \}           \\
      &    &                                             \\
\fM '' & := & \{ L_{i_1}'' \cdots L_{i_n}''  \mid
n \geq 1, \ 1 \leq i_1, \ldots , i_n \leq k \} .
\end{array}   \right.
\end{equation}
\end{notation}

$\ $

The full Fock space model from Theorem \ref{thm:4.4} will use 
some special monomials ``$S_1^* M_1 \cdots$
$S_n^* M_n$'' formed with the isometries 
$L_1 ', \ldots , L_k ', L_1 '', \ldots , L_k ''$ and their 
adjoints, which are described in the next lemma.

\begin{lemma}    \label{lemma:4.2}
Given a positive integer $n$ and some fixed indices 
$i_1, \ldots , i_n \in \{ 1, \ldots , k \}$.

$1^o$ Let $\pi$ be a partition in $NC(n)$ and let 
$c: \pi \to \{ 1,2 \}$ be a colouring. For every 
$m \in \{ 1, \ldots , n \}$ let $V = \{ v(1), v(2), \ldots , v(p) \}$ 
(with $v(1)< v(2)< \cdots < v(p)$) denote the block of $\pi$ which 
contains $m$, and define
\begin{equation}    \label{eqn:4.21}
S_m := \left\{   \begin{array}{lc} 
{L'}_{i_m},  & \mbox{if $c(V) =1$}     \\
{L''}_{i_m}, & \mbox{if $c(V) =2$,}    
\end{array}  \right.
\end{equation} 
\begin{equation}    \label{eqn:4.22}
M_m = \left\{   \begin{array}{ll} 
{L'}_{i_{v(p)}}  \cdots {L'}_{i_{v(2)}}  {L'}_{i_{v(1)}},     &
  \mbox{if $m = \max (V) \bigl( = v(p) \bigr)$ and $c(V) =1$}     \\
{L''}_{i_{v(p)}}  \cdots {L''}_{i_{v(2)}}  {L''}_{i_{v(1)}},  &
  \mbox{if $m = \max (V)$ and $c(V) =2$}                          \\
1_{ B( \cF ) }   &  \mbox{if $m \neq \max (V)$.}
\end{array}  \right.
\end{equation} 
Then $S_1^* M_1 \cdots S_n^* M_n \Omega = \Omega$.

$2^o$ Suppose that $S_1, \ldots , S_n, M_1, \ldots , M_n \in B( \cF )$
are such that 

(i) $S_m \in \{ L_{i_m}' , L_{i_m}'' \} , \ \ 1 \leq m \leq n$;

(ii) $M_m \in \{ 1_{ B( \cF ) } \} \cup \fM ' \cup \fM '', \ \ 
1 \leq m \leq n$ (with $\fM ', \fM ''$ as in (\ref{eqn:4.12})); and

(iii) $S_1^* M_1 \cdots S_n^* M_n \Omega = \Omega$.

\noindent
Then there exist a partition $\pi \in NC(n)$ and a colouring 
$c: \pi \to \{ 1,2 \}$ such that $S_1, \ldots ,S_n$, 
$M_1, \ldots ,M_n$ are obtained from $\pi$ and $c$ via the recipe 
described in part $1^o$ of the lemma.
\end{lemma}

\begin{remark}    \label{rem:4.3}

$1^o$ Here is a concrete example of how the recipe from Lemma 
\ref{lemma:4.2} works. Say for instance that $n=5$. Let 
$i_1, \ldots , i_5$ be some indices in $\{ 1, \ldots , k \}$,
and consider the monomial
\begin{equation}    \label{eqn:4.31}
( L_{i_1}' )^{*} \, ( L_{i_2}'' )^{*} \, ( L_{i_3}'' )^{*}
\, L_{i_3}'' \, L_{i_2}'' 
\, ( L_{i_4}' )^{*} \, ( L_{i_5}' )^{*} 
\, L_{i_5}'  \, L_{i_4}' \, L_{i_1}' .
\end{equation}  
Note that the product in (\ref{eqn:4.31}) reduces upon 
simplifications to $1_{ B( \cF ) }$, so in particular it fixes 
$\Omega$. Lemma \ref{lemma:4.2} views this product as being
$S_1^* M_1 \cdots S_5^* M_5$, where
\[
\left\{   \begin{array}{l}
S_1 = L_{i_1}', \, S_2 = L_{i_2}'', \, S_3 = L_{i_3}'', \,
S_4 = L_{i_4}', \, S_5 = L_{i_5}',  \ \ \mbox{ and }        \\
\mbox{  }                                                   \\
M_1 = M_2 = M_4 = 1_{ B( \cF ) }, \
M_3 = L_{i_3}'' L_{i_2}'', \  
M_5 = L_{i_5}' L_{i_4}' L_{i_1}'.
\end{array}  \right.
\]
Moreover, these $S_1, \ldots , S_5, M_1, \ldots , M_5$ correspond
in Lemma \ref{lemma:4.2} to the partition 
$\pi = \{ V_1, V_2 \} \in NC(5)$ with $V_1 = \{ 1,4,5 \}$,
$V_2 = \{ 2,3 \}$, and to the colouring $c: \pi \to \{ 1,2 \}$ 
defined by $c(V_1) =1$, $c(V_2) = 2$.

$2^o$ The proof of Lemma \ref{lemma:4.2} is very similar to the 
corresponding argument concerning the standard full Fock space 
model for the $R$-transform, as presented e.g. in Lecture 21 of
\cite{NS06}. Because of this, I will only explain (in the 
remaining part of this remark) how one makes the connection to 
the arguments from \cite{NS06}, and will leave the details as 
exercise to the reader. 

Besides $\fM '$ and $\fM ''$ from (\ref{eqn:4.12}), let us also
use the notation
\begin{equation}    \label{eqn:4.32}
\fM := \{ 1_{ B( \cF ) } \} \cup  \Bigl\{ S_1 \cdots S_{\ell} \mid 
\ell \geq 1, \ S_1, \ldots , S_{\ell} \in \{ L_1, \ldots , L_k', 
L_1'', \ldots , L_k'' \} \, \Bigr\} .
\end{equation}
Suppose that the following data is given: a positive integer $n$,
some indices $i_1, \ldots , i_n \in \{ 1, \ldots , k \}$, and a 
function $b: \{ 1, \ldots , n \} \to \{ 1,2 \}$. Let the isometries
$S_1 \in \{ L_{i_1}', L_{i_1}'' \} , \ldots ,$
$S_n \in \{ L_{i_n}', L_{i_n}'' \}$ be picked via the rule that 
\begin{equation}    \label{eqn:4.33}
S_m = \left\{  \begin{array}{ll}
L_{i_m}',   &  \mbox{if $b(m)=1$}     \\
L_{i_m}'',  &  \mbox{if $b(m)=2$,} 
\end{array}   \right. \ \ 1 \leq m \leq n,
\end{equation}
and consider the following problem: describe all possible ways of
choosing $( M_1, \ldots , M_n) \in \fM^n$ such that 
$S_1^{*}M_1 \cdots S_n^* M_n \Omega = \Omega$.
\footnote{ It is easy to see that that this condition is in fact 
equivalent to the requirement that the product 
$S_1^{*}M_1 \cdots S_n^* M_n$ simplifies to $1_{ B( \cF ) }$ 
after repeated use of the relations 
$( L_i' )^* L_i' = ( L_i'' )^* L_i'' = 1_{ B( \cF ) }$, 
$1 \leq i \leq k$. } 
The solution to this problem is that the $n$-tuples 
$(M_1, \ldots , M_n)$ with the required property are canonically 
parametrized by $NC(n)$. For the description of how to construct 
the $n$-tuple $(M_1, \ldots , M_n) \in \fM^n$ canonically associated 
to a partition $\pi \in NC(n)$, and for the explanation why this 
construction works, see the discussion on pp. 342-343 and the 
Exercises 21.20-21.22 on pp. 356-357 of \cite{NS06}. The statement
of Lemma \ref{lemma:4.2} is merely an adjustment of this procedure
(for how to construct $(M_1, \ldots , M_n)$ by starting from $\pi$),
where one has to take into account the following additional detail: 
$M_1, \ldots , M_n$ are now only allowed to run in the smaller set
$\{ 1_{ B( \cF ) } \} \cup \fM ' \cup \fM ''$ (instead of all of
$\fM$). This imposes a compatibility condition between $\pi$ and 
the function $b : \{ 1, \ldots , n \} \to \{ 1,2 \}$ that was used
in (\ref{eqn:4.33}) -- specifically, that $b$ must be constant along
the blocks of $\pi$ (and hence must correspond to a colouring $c$
of $\pi$).
\end{remark}

\begin{theorem}   \label{thm:4.4}
Let $\mu , \nu$ be distributions in $\dalg (k)$ such that the 
$R$-transforms $R_{\mu}$ and $R_{\nu}$ are polynomials:
\begin{equation}   \label{eqn:4.41}
\left\{  \begin{array}{lcl}
R_{\mu} (z_1, \ldots , z_k) & = &
\sum_{n=1}^N \sum_{i_1, \ldots , i_n = 1}^k \
\alpha_{ (i_1, \ldots , i_n) } z_{i_1} \cdots z_{i_n}   \\ 
                            &   &                       \\
R_{\nu} (z_1, \ldots , z_k) & = &
\sum_{n=1}^N \sum_{i_1, \ldots , i_n = 1}^k \
\beta_{ (i_1, \ldots , i_n) } z_{i_1} \cdots z_{i_n}   
\end{array}  \right.
\end{equation}
(where $N$ is a common upper bound for the degrees of 
$R_{\mu}$ and $R_{\nu}$). In the framework of Notation \ref{def:4.1},
consider the operator $T \in B( \cF )$ defined by
\begin{equation}   \label{eqn:4.42}
T = 1_{B( \cF )} + 
\sum_{n=1}^N \ \sum_{i_1, \ldots , i_n = 1}^k \
\alpha_{ (i_1, \ldots , i_n) } L_{i_n} ' \cdots L_{i_1}'   
+ \sum_{n=1}^N \ \sum_{i_1, \ldots , i_n = 1}^k \
\beta_{ (i_1, \ldots , i_n) } L_{i_n} '' \cdots L_{i_1} ''  ,
\end{equation}
and make the notations 
\begin{equation}   \label{eqn:4.43}
A_i := ( L_i ' )^* T, \ \ B_i := ( L_i '' )^* T, \ \ 
1 \leq i \leq k,
\end{equation}
followed by 
\begin{equation}   \label{eqn:4.44}
C_i := A_i + (1- \pvac ) \, B_i \, (1- \pvac ), \ \ 1 \leq i \leq k.
\end{equation}
Then the joint distribution of $C_1, \ldots , C_k$ with respect to
the vacuum-state on $B( \cF )$ is equal to $\mu \boxright \nu$.
\end{theorem}

\begin{remark}   \label{rem:4.45}
By comparing the framework of Theorem \ref{thm:4.4} with the 
``standard'' full Fock space model for the $R$-transform (as 
presented for instance in Theorem 21.4 of \cite{NS06}), one 
sees that the operators $A_1, \ldots , A_k$,
$B_1, \ldots , B_k$ defined by Equation (\ref{eqn:4.43}) give 
the standard full Fock space model for the free product 
$\mu \star \nu$. In particular one has that 
$\{ A_1, \ldots , A_k \}$ is free from $\{ B_1, \ldots ,B_k \}$
with respect to the vacuum-state on $B( \cF )$, and the joint 
distributions of the $k$-tuples $A_1, \ldots , A_k$ and 
$B_1, \ldots , B_k$ are equal to $\mu$ and to $\nu$, respectively.

An other way of phrasing this same remark is that the full Fock 
space model for $\mu \boxright \nu$ is obtained by merely performing
an extra step (specifically, by considering the operators 
$C_1, \ldots , C_k$ defined by Equation (\ref{eqn:4.44})) in the 
standard full Fock space model for $\mu \star \nu$.
\end{remark}

$\ $

{\em Proof of Theorem \ref{thm:4.4}.} For the whole proof we fix 
a positive integer $n$ and some indices 
$1 \leq i_1, \ldots , i_n \leq k$, for which we will show that 
\begin{equation}   \label{eqn:4.45}
\langle \, C_{i_1} \cdots C_{i_n} \Omega \, , \, \Omega \, \rangle
= ( \mu \boxright \nu ) (X_{i_1} \cdots X_{i_n}).
\end{equation}

$ $From (\ref{eqn:4.42})--(\ref{eqn:4.44}) it follows that every 
$C_i$ $(1 \leq i \leq k)$ can be written as a sum of products of
the form
\begin{equation}   \label{eqn:4.46}
Q \cdot S^* \cdot ( \gamma M ) \cdot Q,
\end{equation}
where $Q \in \{ 1_{B( \cF )}, 1_{B( \cF )} - \pvac \}$,
$S \in \{ L_i ', L_i '' \}$, and $\gamma M$ is a term from the sum
defining $T$ (where $\gamma \in \bC$ and 
$M \in \{ 1_{B( \cF )} \} \cup \fM ' \cup \fM ''$). Of course, there 
are some restrictions on what combinations of $Q, S$ and $\gamma M$ 
can go together in (\ref{eqn:4.46}): if $Q = 1_{B( \cF )}$ then 
$S= L_i '$ and $\gamma M$ is either $1_{B( \cF )}$ or of the form 
$\alpha_{ (j_1, \ldots , j_m) } L'_{j_m} \cdots L'_{j_1}$, while
$Q = 1_{B( \cF )} - \pvac$ goes with $S= L_i ''$ and with $\gamma M$ 
being either $1_{B( \cF )}$ or of the form 
$\beta_{ (j_1, \ldots , j_m) } L''_{j_m} \cdots L''_{j_1}$. A 
precise count thus gives that every $C_i$ splits into a sum of 
$2 \cdot (1+k+ \cdots +k^N)$ terms of the form (\ref{eqn:4.46}).
When one writes each of $C_{i_1}, \ldots , C_{i_n}$ as a sum in 
this way and expands the product, the inner product on the left-hand 
side of (\ref{eqn:4.45}) is thus broken into a sum of  
$\bigl( 2 \cdot (1+k+ \cdots +k^N) \bigr)^n$ terms of the form 
\begin{equation}   \label{eqn:4.47}
\langle \, (Q_1 {S_1}^* ( \gamma_1 M_1 ) Q_1) \cdots
(Q_n {S_n}^* ( \gamma_n M_n ) Q_n) \Omega \, , \, \Omega \, \rangle .
\end{equation}

Now let us fix one of the possible choices of operators 
$Q_i, S_i, M_i$ $(1 \leq i \leq n)$ in (\ref{eqn:4.47}), and let us
look at the $4n$ vectors 
\begin{equation}    \label{eqn:4.475}
\xi_1 = Q_n \Omega, \ \xi_2 = M_n Q_n \Omega, \ldots ,
\xi_{4n} = Q_1 S_1^* M_1 Q_1  \cdots Q_n S_n^* M_n Q_n \Omega
\end{equation}
obtained by succesively applying the operators 
$Q_n, M_n, {S_n}^*, Q_n, \ldots , Q_1, M_1, {S_1}^*, Q_1$ to 
$\Omega$. It is clear that each of these $4n$ vectors either is 0 
or belongs to the orthonormal basis (\ref{eqn:4.11}) for $\cF$;
and consequently, the inner product (\ref{eqn:4.47}) is equal to
\begin{equation}    \label{eqn:4.48}
\left\{  \begin{array}{ll}
\gamma_1 \cdots \gamma_n,  & \mbox{if }
Q_1 S_1^* M_1 Q_1 \cdots Q_n S_n^* M_n Q_n \Omega = \Omega    \\
0,                         & \mbox{otherwise.}
\end{array}   \right.
\end{equation}
Let us moreover observe that if 
$Q_1 S_1^* M_1 Q_1 \cdots Q_n S_n^* M_n Q_n \Omega = \Omega$, 
then we also have 
$S_1^* M_1 \cdots S_n^* M_n \Omega = \Omega$. This is because 
when one succesively applies $Q_n, M_n, \ldots , {S_1}^*, Q_1$ 
to $\Omega$, the projections $Q_1, \ldots , Q_n$ used on the way
either leave invariant the vector presented to them, or send it 
to 0 (but can't actually do the latter, as 
$Q_1 S_1^* \cdots M_n Q_n \Omega = \Omega \neq 0$). 

By invoking Lemma \ref{lemma:4.2} we thus see that if an inner 
product as in (\ref{eqn:4.47}) is to be different from 0, then 
there have to exist a partition $\pi \in NC(n)$ and a colouring 
$c: \pi \to \{ 1,2 \}$ such that $S_1, M_1, \ldots , S_n, M_n$
are defined in terms of $\pi$ and $c$ in the way described in 
Lemma \ref{lemma:4.2}. It is immediate that in this case 
the numbers $\gamma_1, \ldots , \gamma_n$ from (\ref{eqn:4.48}) 
are identified as $\alpha_{ (j_1, \ldots, j_m) }$'s and  
$\beta_{ (j_1, \ldots, j_m) }$'s (coefficients of the 
$R$-transforms of $\mu$ and of $\nu$) in such a way that
their product becomes
\begin{equation}     \label{eqn:4.410}
\gamma_1 \cdots \gamma_n = 
\cf_{ (i_1, \ldots , i_n); \pi ; c } ( R_{\mu} , R_{\nu} ).
\end{equation}

Conversely, let $\pi$ be a partition in $NC(n)$, let $c$ be a 
colouring of $\pi$, and consider the operators 
$S_1, M_1, \ldots , S_n, M_n$
defined in terms of $\pi$ and $c$ in the way described in 
Lemma \ref{lemma:4.2}. Observe that there exists a unique way of
choosing projections $Q_1, \ldots , Q_n \in 
\{ 1_{B( \cF )}, 1_{B( \cF )} - P_{\Omega} \}$ so that 
the $S_j, M_j, Q_j$ for $1 \leq j \leq n$ give together an 
inner product as in (\ref{eqn:4.47}). To be precise, for every 
$1 \leq j \leq n$ the projection $Q_j$ is chosen as follows: 
consider the block $V$ of $\pi$ which contains the number $j$, 
and put
\begin{equation}    \label{eqn:4.411}
Q_j = \left\{   \begin{array}{ll}
1_{B( \cF )},              & \mbox{if $c(V) = 1$}    \\
1_{B( \cF )} - P_{\Omega}, & \mbox{if $c(V) = 2$.} 
\end{array}   \right.
\end{equation}
Note that whereas Lemma \ref{lemma:4.2} ensures that
$S_1^* M_1 \cdots S_n^* M_n \Omega = \Omega$,
it may still happen that (with $Q_j$s defined by (\ref{eqn:4.411})) 
the vector $Q_1 S_1^* M_1 Q_1 \cdots Q_n S_n^* M_n Q_n \Omega$
is equal to 0. It as easy (though perhaps notationally tedious) to 
check that 
\begin{equation}    \label{eqn:4.412}
Q_1 S_1^* M_1 Q_1 \cdots Q_n S_n^* M_n Q_n \Omega =
\left\{   \begin{array}{ll}
\Omega,  & \mbox{if $c(V)=1$ for every outer block of $\pi$}   \\
0,       & \mbox{otherwise.}
\end{array}   \right.
\end{equation}
The verification of (\ref{eqn:4.412}) is left as exercise to the 
reader. Informally speaking, what makes (\ref{eqn:4.412}) hold is 
that in a sequence of $4n$ vectors obtained as in (\ref{eqn:4.475}) 
one reaches $\Omega$ 
precisely at the positions where the outer blocks of $\pi$ begin
and end -- hence these are the positions where a $Q_j$ has a chance 
to make a difference, and cause the vector
$Q_1 S_1^* M_1 Q_1 \cdots Q_n S_n^* M_n Q_n \Omega$ to vanish.

Summarizing the above discussion, one sees that 
\begin{equation}   \label{eqn:4.413}
\langle \, C_{i_1} \cdots C_{i_n} \Omega \, , \, \Omega \, \rangle
= \sum_{ ( \pi , c ) }
\cf_{ (i_1, \ldots , i_n); \pi ; c } ( R_{\mu} , R_{\nu} ) ,
\end{equation}   
where the index set for the 
sum on the right-hand side of (\ref{eqn:4.413}) is 
\[
\left\{  ( \pi , c ) \
\begin{array}{cl}
\vline & \pi \in NC(n), \ c: \pi \to \{ 1,2 \}, \mbox{ such that} \\
\vline & \mbox{$c(V)=1$ for every outer block $V$ of $\pi$}
\end{array} \right\}  .
\]
But the sum on the right-hand side of (\ref{eqn:4.413}) is 
precisely the expression observed for 
$( \mu \boxright \nu) ( X_{i_1} \cdots X_{i_n} )$ in 
Remark \ref{rem:3.10}.2, and this concludes the proof.
\hfill $\blacksquare$

$\ $

Let us now go towards the proof of Theorem \ref{thm:1.4}. It will 
be convenient to adopt a slightly different point of view for the 
projection onto the vacuum-vector, which does not make explicit
use of vectors and operators, and is described as follows.

\begin{definition}   \label{def:4.5}
Let $( \cA , \varphi )$ be a noncommutative probability space. 
A {\em vacuum-projection} for $\varphi$ is an element $P \in \cA$
such that $P=P^2 \neq 0$ and such that 
\begin{equation}   \label{eqn:4.51}
PAP = \varphi (A) P, \ \ \forall \, A \in \cA .
\end{equation}
\end{definition}

\begin{remark}  \label{rem:4.6} 

$1^o$ The main example of vacuum-projection is of course provided
by the situation when $\cA = B( \cH )$, the functional 
$\varphi$ is the vector-state associated to a unit vector 
$\xi_0 \in \cH$, and $P$ is the orthogonal projection onto the 
1-dimensional subspace $\bC \xi_0$ of $\cH$.

$2^o$ Let $( \cA , \varphi )$ and $P$ be as in Definition 
\ref{def:4.5}. Observe that $\varphi (P) =1$ (as seen
by making $A=P$ in Equation (\ref{eqn:4.51})). Let us 
also observe that 
\begin{equation}   \label{eqn:4.61}
\varphi (PB) = \varphi (B) = \varphi (BP), \ \ 
\forall \, B \in \cA .
\end{equation}
In order to verify the first of these two equalities 
we set $A = (1_{\cA} -P)B$ and find that 
\[
\varphi (A) P = PAP = P (1_{\cA} -P)B P = 0,
\]
which implies that $\varphi (A) = 0$ and hence that 
$\varphi (B) = \varphi (PB)$. The verification of the second 
equality in (\ref{eqn:4.61}) is analogous.
\end{remark}

\begin{lemma}   \label{lemma:4.7}
Let $( \cA , \varphi )$ be a noncommutative probability space 
and let $P \in \cA$ be a vacuum-projection for $\varphi$. Then
\begin{equation}    \label{eqn:4.71}
\varphi (T_1 P T_2 P \cdots P T_n ) = \prod_{i=1}^n \varphi (T_i),
\ \ \forall \, n \geq 2 \mbox{ and } T_1, \ldots , T_n \in \cA.
\end{equation}
\end{lemma}

\begin{proof} By induction on $n$. For $n=2$ we write
\begin{align*}
\varphi (T_1 P T_2 ) 
& = \varphi (T_1 P T_2 P)  &  \mbox{ (by (\ref{eqn:4.61})) }   \\
& = \varphi (T_1 \cdot \varphi (T_2) P)  
                           &  \mbox{ (by (\ref{eqn:4.51})) }   \\ 
& = \varphi (T_2) \ \varphi (T_1 P)  
                           &                                   \\
& = \varphi (T_2) \ \varphi (T_1) 
                           & \mbox{ (by (\ref{eqn:4.61})). }  
\end{align*}
The induction step ``$n \Rightarrow n+1$'' is immediately obtained
by writing $T_1 P T_2 P \cdots P T_n P T_{n+1}$ as $T_1 P T_2'$
with $T_2' := T_2 P \cdots P T_n P T_{n+1}$ and by repeating the 
above calculation, followed by the induction hypothesis.
\end{proof}

\begin{lemma}   \label{lemma:4.8}
Let $( \cA , \varphi )$ be a noncommutative probability space and 
and let $T_1, \ldots , T_{\ell}, P$ be in $\cA$, where 
$P$ is a vacuum-projection for $\varphi$. Suppose moreover that 
for every $N \geq 1$ we are given a noncommutative probability 
space $( \cA_N, \varphi_N )$ and elements 
$T_1^{(N)}, \ldots , T_{\ell}^{(N)}, P^{(N)} \in \cA_N$,
such that $P^{(N)}$ is a vacuum-projection for $\varphi_N$. If 
the $\ell$-tuples $T_1^{(N)}, \ldots , T_{\ell}^{(N)}$
converge in distribution for $N \to \infty$ to 
$T_1, \ldots , T_{\ell}$, then the $( \ell + 1)$-tuples 
$T_1^{(N)}, \ldots , T_{\ell}^{(N)}, P^{(N)}$
converge in distribution for $N \to \infty$ to the 
$( \ell +1 )$-tuple $T_1, \ldots , T_{\ell}, P$.
\end{lemma}

\begin{proof} It clearly suffices to verify that, for any 
$n \geq 2$ and any choice of non-commutative polynomials
$f_1, \ldots , f_n \in \bC \langle X_1, \ldots , X_{\ell} \rangle$,
the sequence 
\[
\varphi_N \Bigl( \, 
f_1 ( T_1^{(N)}, \ldots , T_{\ell}^{(N)} ) P^{(N)}
f_2 ( T_1^{(N)}, \ldots , T_{\ell}^{(N)} ) P^{(N)} \cdots
P^{(N)} f_n ( T_1^{(N)}, \ldots , T_{\ell}^{(N)} ) \, \Bigr), 
\ \ N \geq 1
\]
converges for $N \to \infty$ to
$\varphi \bigl( \, f_1 ( T_1, \ldots , T_{\ell} ) P
f_2 ( T_1, \ldots , T_{\ell} ) P \cdots P 
f_n ( T_1, \ldots , T_{\ell} ) \, \bigr).$ 
But in view of Lemma \ref{lemma:4.7} the latter convergence 
amounts to 
\[
\lim_{N \to \infty} \prod_{i=1}^n 
\varphi_N ( \, f_i ( T_1^{(N)}, \ldots , T_{\ell}^{(N)} ) \, ) = 
\prod_{i=1}^n \varphi ( \, f_i ( T_1, \ldots , T_{\ell} ) \, ),
\]
which is an immediate consequence of the given hypothesis.
\end{proof}

\begin{theorem}   \label{thm:4.9}
Let two distributions $\mu , \nu \in \dalg (k)$ be given. Suppose that 
$( \cA , \varphi )$ is a noncommutative probability space and that
$A_1, \ldots , A_k, B_1, \ldots , B_k \in \cA$ are such that 
$\{ A_1, \ldots , A_k \}$ is free from $\{ B_1, \ldots , B_k \}$,
such that the joint distribution of $A_1, \ldots , A_k$ is equal 
to $\mu$, and such that the joint distribution of $B_1, \ldots ,B_k$ 
is equal to $\nu$. Suppose in addition that $P \in \cA$ is a 
vacuum-projection for $\varphi$, and consider the elements 
\begin{equation}    \label{eqn:4.91}
C_i := A_i + ( 1_{\cA} -P) B_i (1_{\cA} -P), \ \ 
1 \leq i \leq k.
\end{equation}
Then the joint distribution of $C_1, \ldots , C_k$ with respect
to $\varphi$ is equal to $\mu \boxright \nu$.
\end{theorem}

\begin{proof} For $n \geq 1$ and $1 \leq i_1, \ldots , i_n \leq k$
we will denote the coefficients of $z_{i_1} \cdots z_{i_n}$ in the 
series $R_{\mu}$ and $R_{\nu}$ by $\alpha_{ (i_1, \ldots , i_n) }$
and $\beta_{ (i_1, \ldots , i_n) }$, respectively.

Let $N$ be a positive integer. Consider the distributions 
$\mu_N, \nu_N \in \dalg (k)$ which are uniquely determined by the 
requirement that their $R$-transforms are 
\begin{equation}    \label{eqn:4.92}
\left\{  \begin{array}{lcl}
R_{\mu} (z_1, \ldots , z_k) & = &
\sum_{n=1}^N \sum_{i_1, \ldots , i_n = 1}^k \
\alpha_{ (i_1, \ldots , i_n) } z_{i_1} \cdots z_{i_n}   \\ 
                            &   &                       \\
R_{\nu} (z_1, \ldots , z_k) & = &
\sum_{n=1}^N \sum_{i_1, \ldots , i_n = 1}^k \
\beta_{ (i_1, \ldots , i_n) } z_{i_1} \cdots z_{i_n}  .
\end{array}  \right.
\end{equation}  
Let us consider the standard full Fock space model, exactly as 
described in Theorem 21.4 of \cite{NS06}, for the free product 
$\mu_N * \nu_N \in \dalg (2k)$. This gives us a 
noncommutative probability space $( \cA_N, \varphi_N )$ and 
elements $A_1^{(N)}, \ldots , A_k^{(N)}$, 
$B_1^{(N)},\ldots , B_k^{(N)} \in \cA_N$  such that 
$\{ A_1^{(N)}, \ldots , A_k^{(N)} \}$ is free from 
$\{ B_1^{(N)}, \ldots , B_k^{(N)} \}$, such that the joint 
distribution of $A_1^{(N)}, \ldots , A_k^{(N)}$ is equal 
to $\mu_N$, and such that the joint distribution of 
$B_1^{(N)}, \ldots ,B_k^{(N)}$ is equal to $\nu_N$. Since the full 
Fock space model is constructed by using a true vacuum-state on a
Hilbert space, we also get at the same time a vacuum-projection 
$P^{(N)} \in \cA_N$.

We now make $N \to \infty$. From how $\mu_N$ and $\nu_N$ were 
constructed it is immediate that we have limits in distribution 
$\mu_N \to \mu$ and $\nu_N \to \nu$. This implies that we also 
have the limit in distribution $\mu_N * \nu_N \to \mu * \nu$,
or in terms of operators that the $(2k)$-tuples 
$A_1^{(N)}, \ldots , A_k^{(N)}, B_1^{(N)}, \ldots , B_k^{(N)}$
converge in distribution for $N \to \infty$ to the $(2k)$-tuple
$A_1, \ldots , A_k, B_1, \ldots , B_k$. By invoking Lemma 
\ref{lemma:4.8} we upgrade this to the fact that 
the $( 2k + 1)$-tuples $A_1^{(N)}, \ldots , A_k^{(N)}$, 
$B_1^{(N)}, \ldots , B_{k}^{(N)}, P^{(N)}$
converge in distribution for $N \to \infty$ to the 
$( 2k +1 )$-tuple $A_1, \ldots , A_k, B_1, \ldots , B_k, P$. 
The latter convergence implies in turn that the 
$k$-tuple $C_1 \ldots , C_k$ defined in (\ref{eqn:4.91}) is the 
limit in distribution for the $k$-tuples 
$C_1^{(N)}, \ldots , C_k^{(N)}$, where for $1 \leq i \leq k$
and $N \geq 1$ we put
\begin{equation}    \label{eqn:4.93}
C_i^{(N)} := A_i^{(N)} + ( 1_{\cA_N} - P^{(N)} ) B_i^{(N)}
( 1_{\cA_N} - P^{(N)} ) \in \cA_N.
\end{equation}

But for every $N \geq 1$, the operators 
$C_1^{(N)}, \ldots , C_k^{(N)}$ provide (as observed at the end 
of Remark \ref{rem:4.45}) the full Fock space model for 
the subordination distribution $\mu_N \boxright \nu_N$. Hence 
the conclusion of the preceding paragraph can be read as follows:
the joint distribution of $C_1, \ldots , C_k$ is the 
$N \to \infty$ limit of the distributions $\mu_N \boxright \nu_N$. 
Since it was noticed in Remark \ref{rem:3.11} that 
$( \, \mu_N \boxright \nu_N \, )_{N=1}^{\infty}$ converges in
distribution to $\mu \boxright \nu$, the conclusion of the theorem 
follows.
\end{proof}

\begin{remark}   \label{rem:4.10}  
Suppose now that $\mu , \nu \in \cD_c (k)$, i.e. they can 
appear as joint distributions for $k$-tuples of selfadjoint 
elements in some $C^*$-probability spaces. By considering the 
GNS representations of these $C^*$-probability spaces, one 
finds Hilbert spaces $\cH  , \cK$, unit vectors 
$\xi_o \in \cH$, $\zeta_o \in \cK$, and $k$-tuples of selfadjoint 
operators $A_1, \ldots , A_k \in B( \cH )$, 
$B_1, \ldots , B_k \in B( \cK )$ such that $\mu$ is the joint
distribution of $A_1, \ldots , A_k$ with respect to the vector-state
defined by $\xi_o$ on $B( \cH )$, while $\nu$ is the joint 
distribution of $B_1, \ldots , B_k$ with respect to the vector-state
defined by $\zeta_o$ on $B( \cK )$. Let us denote
\[
\cH^{o} := \cH \ominus \bC \xi_o, \ \
\cK^{o} := \cK \ominus \bC \zeta_o, 
\]
and let us consider the ``free product'' Hilbert space
\[
\cM := \bC \Omega \oplus 
\Bigl( \, \cH^{o} \oplus \cK^{o} \, \Bigr)
\oplus \Bigl( \, ( \cH^o \otimes \cK^o ) \oplus
( \cK^o \otimes \cH^o ) \, \Bigr) 
\]
\begin{equation}    \label{eqn:4.101}
{  }  \hspace{2cm}
\oplus \Bigl( \, ( \cH^o \otimes \cK^o \otimes \cH^o )
\oplus ( \cK^o \otimes \cH^o \otimes \cK^o ) \, \Bigr) \oplus \cdots
\end{equation}
(direct sum of all possible alternating tensor products of copies 
of $\cH^o$ and $\cK^o$). Then $A_1, \ldots , A_k, B_1, \ldots , B_k$
extend naturally to selfadjoint operators
$\widetilde{A}_1, \ldots , \widetilde{A}_k$,
$\widetilde{B}_1, \ldots , \widetilde{B}_k \in B( \cM )$ such that
$\{ \widetilde{A}_1, \ldots , \widetilde{A}_k \}$ is free from
$\{ \widetilde{B}_1, \ldots , \widetilde{B}_k \}$ with respect to
the vacuum-state defined by $\Omega$ on $B( \cM )$ and such that
(with respect to the same state) the joint distributions of 
$\widetilde{A}_1, \ldots , \widetilde{A}_k$ and of
$\widetilde{B}_1, \ldots , \widetilde{B}_k$ are equal to $\mu$
and $\nu$, respectively (see e.g. \cite{VDN92}, Section 1.5).

Theorem \ref{thm:4.9} clearly applies in the situation described 
in the preceding paragraph, and tells us that if 
$P_{\Omega} \in B( \cM )$ is the orthogonal projection onto 
$\bC \Omega$ and if we put
\begin{equation}    \label{eqn:4.102}
\widetilde{C}_i = \widetilde{A}_i + 
(1- P_{\Omega}) \, \widetilde{B}_i \, (1- P_{\Omega}), \ \ 
1 \leq i \leq k,
\end{equation}
then the joint distribution of 
$\widetilde{C}_1, \ldots , \widetilde{C}_k$ is equal to 
$\mu \boxright \nu$. Since the $\widetilde{C}_i$ are 
selfadjoint, this provides us with a proof that (as stated in 
Corollary \ref{cor:1.6} of the introduction) the subordination 
distribution $\mu \boxright \nu$ does indeed belong to $\cD_c (k)$.
\end{remark}

\begin{remark}   \label{rem:4.11}
In the framework and notations of the preceding remark, 
consider the subspace $\cL$ of $\cM$ defined by:
\begin{equation}    \label{eqn:4.103}
\cL := \bC \Omega \oplus \cH^{o} \oplus 
\bigl( \, \cK^o \otimes \cH^o \, \bigr) \oplus 
\bigl( \, \cH^o \otimes \cK^o \otimes \cH^o \, \bigr) \oplus \cdots
\end{equation}
(direct sum of all alternating tensor products of copies 
of $\cH^o$ and $\cK^o$ which end in $\cH^o$). In the terminology 
of \cite{L07}, this is the {\em $s$-free product space} of 
the Hilbert spaces $\cH$ and $\cK$, considered with respect to the 
special unit vectors $\xi_o \in \cH$ and $\zeta_o \in \cK$.

Observe that $\cL$ is invariant for the operators
$\widetilde{C}_1, \ldots , \widetilde{C}_k$ from (\ref{eqn:4.102}); 
this happens because $\cL$ is in fact invariant both for
$\widetilde{A}_i$ and for
$(1-P_{\Omega}) \,  \widetilde{B}_i \, (1-P_{\Omega} )$,
$1 \leq i \leq k$. It follows that the restrictions of  
$\widetilde{C}_1, \ldots , \widetilde{C}_k$ to $\cL$ also provide
us with an operator model for $\mu \boxright \nu$, with respect to
the vector-state defined by $\Omega$ on $B( \cL )$. By analyzing 
this operator model a bit further, one can moreover relate to the 
concept of ``$s$-freeness'' from \cite{L07}, in the way outlined 
in the next paragraph.

For every $1 \leq i \leq k$ let 
$\widehat{A}_i$ and $\widehat{B}_i$ denote the restrictions to $\cL$
of the operators $\widetilde{A}_i$ and respectively 
$(1-P_{\Omega}) \,  \widetilde{B}_i \, (1-P_{\Omega} )$. Let us 
consider the subalgebras $\cA , \cB$ of $B( \cL )$ which are 
generated by 
$\{ 1_{B( \cL )}, \widehat{A}_1, \ldots , \widehat{A}_k \}$ 
and respectively by $\{ 1_{ B( \cL ) } - P_{\Omega}, 
\widehat{B}_1, \ldots , \widehat{B}_k \}$. (Note that $\cB$ is not 
a unital subalgebra of $B( \cL )$, but it has its own unit 
$1_{\cB} = 1_{ B( \cL ) } - P_{\Omega}$, where $P_{\Omega}$ is 
viewed here as a 1-dimensional projection in $B( \cL )$.) Finally,
let us select (and fix) an arbitrary unit vector 
$\theta_o \in \cH^o \subseteq \cL$, and let $\varphi$ and $\psi$ 
be the vector-states defined on $B( \cL )$ by $\Omega$ and by 
$\theta_o$, respectively. It is not hard to verify
that the algebras $\cA$ and $\cB$ are $s$-free in 
$( \, B( \cL ), \varphi , \psi \, )$, in the sense of Definition 
7.1 from \cite{L07}. It is moreover immediate that the 
joint distribution of $\widehat{A}_1, \ldots , \widehat{A}_k$ in 
$( \cA \, , \varphi | \cA )$ is equal to $\mu$, while the joint 
distribution of $\widehat{B}_1, \ldots , \widehat{B}_k$ in 
$( \cB \, , \psi | \cB )$ is equal to $\nu$. Thus 
$\mu \boxright \nu$ has been realized as the joint distribution 
of $\widehat{A}_1 + \widehat{B}_1, \ldots , \widehat{A}_k + 
\widehat{B}_k$, where the $k$-tuples 
$\widehat{A}_1, \ldots , \widehat{A}_k$ and
$\widehat{B}_1, \ldots , \widehat{B}_k$ are $s$-free and have 
distributions $\mu$ and $\nu$, respectively. 

The verification of the $s$-freeness of $\cA$ and $\cB$ in 
the preceding paragraph is left as exercise. A reader who is 
interested in $s$-freeness may also find it as an amusing (not
hard) exercise to start from this latter description of 
$\mu \boxright \nu$ and see, conversely, how the statement 
of Theorem \ref{thm:1.4} can be obtained from there.
\end{remark}

\vspace{6pt}

We conclude this section by observing that (as a supplement to 
the fact that $\mu \boxright \nu \in \cD_c (k)$ whenever 
$\mu , \nu \in \cD_c (k)$), there exist natural situations when 
$\mu \boxright \nu$ is sure to be infinitely divisible.

\begin{corollary}    \label{cor:4.11}
Let $\mu, \nu$ be two distributions in $\cD_c (k)$.

$1^o$ If $\mu$ is $\boxplus$-infinitely divisible, then so is
$\mu \boxright \nu$.

$2^o$ Suppose that ``$\mu$ is a $\boxplus$-summand of $\nu$ in 
$\cD_c (k)$'', in the sense that there exists $\nu ' \in \cD_c (k)$
such that $\nu = \mu \boxplus \nu '$. Then $\mu \boxright \nu$ is 
infinitely divisible.
\end{corollary}

\begin{proof} $1^o$ The hypothesis that $\mu$ is 
$\boxplus$-infinitely divisible is equivalent to the fact 
that, for every $t>0$, the convolution power $\mu^{\boxplus t}$
(which can always be defined in $\dalg (k)$) still belongs to 
$\cD_c (k)$. But then, by invoking Remark \ref{rem:1.2}.1 and 
Corollary \ref{cor:1.6} one finds that 
\[
( \mu \boxright \nu )^{\boxplus t} =
( \mu^{\boxplus t} \boxright \nu ) \in \cD_c (k),
\ \ \forall \, t>0,
\]
which means that $\mu \boxright \nu$ is infinitely divisible as well.

$2^o$ One has $\mu \boxright \nu$ = 
$\mu \boxright ( \mu \boxplus \nu ') = \bB ( \mu \boxright \nu ' )$
(where at the second equality sign we used Remark \ref{rem:3.8}.1). 
Since $\mu \boxright \nu ' \in \cD_c (k)$ (by Corollary 
\ref{cor:1.6}), and since $\bB$ carries $\cD_c (k)$ onto the set 
of $\boxplus$-infinitely divisible distributions in $\cD_c (k)$, 
the conclusion follows.
\end{proof}

$\ $

$\ $

\begin{center}
{\bf\large 5. Relations with the transformations $\bB_t$}
\end{center}
\setcounter{section}{5}
\setcounter{equation}{0}
\setcounter{theorem}{0}

\begin{proposition}   \label{prop:5.1}
Let $\mu, \nu$ be distributions in $\dalg (k)$. For every 
$t > 0$ one has that
\begin{equation}   \label{eqn:5.11}  
\bB_t ( \mu \boxright \nu ) = \mu \boxright 
\bigl( \mu^{\boxplus t} \boxplus \nu \bigr).
\end{equation}
\end{proposition}

\begin{proof}
We first prove by induction that 
\begin{equation}   \label{eqn:5.12}  
\bB_m ( \mu \boxright \nu ) = \mu \boxright
\bigl( \mu^{\boxplus m} \boxplus \nu \bigr), \ \ \forall \, 
m \in \bN .
\end{equation}
The base case $m=1$ of the induction is provided by formula 
(\ref{eqn:3.82}) in Remark \ref{rem:3.8}.1. The induction step 
``$m \Rightarrow m+1$'' also follows immediately by using the 
same formula:
\begin{align*}  
\bB_{m+1} ( \mu \boxright \nu )
& = \bB \bigl( \bB_m ( \mu \boxright \nu ) \bigr)  &
      \mbox{ (since $\bB_{m+1} = \bB \circ \bB_m$) }         \\
& = \bB \Bigl( \mu \boxright 
\bigl( \mu^{\boxplus m} \boxplus \nu \bigr) \Bigr)  &
      \mbox{ (by the induction hypothesis) }                 \\
& = \mu \boxright \Bigl( \mu \boxplus 
\bigl( \mu^{\boxplus m} \boxplus \nu \bigr) \Bigr)  &
      \mbox{ (by Equation (\ref{eqn:3.82})) }                \\ 
& = \mu \boxright \Bigl( 
\mu^{\boxplus (m+1)} \boxplus \nu \Bigr) .  &   { }
\end{align*}

Now we move to proving that (\ref{eqn:5.11}) holds for arbitrary 
$t> 0$. It suffices to fix $n \in \bN$ and 
$1 \leq i_1, \ldots , i_n \leq k$ and to verify that 
\begin{equation}   \label{eqn:5.13}  
\cf_{ (i_1, \ldots , i_n) } \Bigl( \, 
R_{ \bB_t ( \mu \boxright \nu ) } \, \Bigr) = 
\cf_{ (i_1, \ldots , i_n) } \Bigl( \, 
R_{ \mu \boxright ( \mu^{\boxplus t} \boxplus \nu ) }
\, \Bigr) , \ \ \forall \, t>0.
\end{equation}
For both sides of (\ref{eqn:5.13}) one has explicit writings as 
sums indexed by non-crossing partitions. Indeed, Remark 4.4 from 
\cite{BN09} tells us that the left-hand side of (\ref{eqn:5.13}) is 
equal to 
\begin{equation}   \label{eqn:5.14}  
\sum_{ \begin{array}{c}
{\scriptstyle  \rho \in NC(n),}    \\
{\scriptstyle  \rho \leqleq 1_n}
\end{array} } \ t^{ | \rho | -1 } \cf_{ (i_1, \ldots , i_n); \rho }
\bigl( R_{\mu \boxright \nu} \bigr) ,
\end{equation}
while the right-hand side of (\ref{eqn:5.13}) can be written (by 
Proposition \ref{prop:3.3} and by taking into account the additivity 
of the $R$-transform) in the form
\begin{equation}   \label{eqn:5.15}  
\sum_{ \begin{array}{c}
{\scriptstyle  \pi \in NC(n),}    \\
{\scriptstyle  \pi \leqleq 1_n}
\end{array} } \ \cf_{ (i_1, \ldots , i_n); \pi ; \oo_{\pi} }
( R_{\mu}, t R_{\mu} + R_{\nu} ).
\end{equation}
Rather than pursuing a detailed combinatorial analysis of the sums
in (\ref{eqn:5.14}) and (\ref{eqn:5.15}) we can simply exploit
the obvious fact that (for our fixed $n$ and $i_1, \ldots , i_n$)
both these sums are polynomial functions of $t$. Two polynomial
functions that agree (as shown by (\ref{eqn:5.12})) for all 
$m \in \bN$ must in fact agree for all $t>0$, and (\ref{eqn:5.13})
follows.
\end{proof}

\begin{remark}   \label{rem:5.2}
As an application of Proposition \ref{prop:5.1}, we will 
next see how the formula ``$\mu \boxright \mu = \bB ( \mu )$'' 
from Remark \ref{rem:1.2}.2 extends to a formula for 
$( \mu^{\boxplus s} ) \boxright ( \mu^{\boxplus t} )$, where
$s,t \geq 0$. In order to cover the cases when $s=0$ or $t=0$,
we will denote by $\ncdirac \in \dalg (k)$ the ``non-commutative 
Dirac distribution at 0'' which has all moments equal to 0. Then,
clearly, $R_{\ncdirac} = \eta_{\ncdirac} = 0 \in 
\bC_0 \langle \langle z_1, \ldots , z_k \rangle \rangle$;
as a consequence one has 
$\delta^{\boxplus t} = \delta^{\uplus t} = \delta$,
hence $\bB_t ( \delta ) = \delta$ for every $t>0$. Moreover, it is 
clear that $\delta$ is the 
neutral element for both the operations $\boxplus$ and $\uplus$ 
on $\dalg (k)$, which justifies the convention that 
\begin{equation}   \label{eqn:5.21}
\mu^{\boxplus 0} = \mu^{\uplus 0} = \delta , \ \ \forall
\, \mu \in \dalg (k).
\end{equation}
Concerning subordination distributions it is easy to check,
directly from Definition \ref{def:1.1}, that 
\begin{equation}   \label{eqn:5.22}
\mu \boxright \delta = \mu \ \mbox{ and } \
\delta \boxright \mu  = \delta, \ \ \forall
\, \mu \in \dalg (k).
\end{equation}
\end{remark}

\vspace{6pt}

\begin{proposition}  \label{prop:5.3}
Let $\mu$ be a distribution in $\dalg (k)$. Then for every 
$s,t \geq 0$ one has
\begin{equation}   \label{eqn:5.31}
\bigl( \mu^{\boxplus s} \bigr) \boxright
\bigl( \mu^{\boxplus t} \bigr) = 
\bigl( \, \bB_t ( \mu ) \, \bigr)^{\boxplus s}.
\end{equation}
\end{proposition}

\begin{proof} First observe that 
\begin{align*}
\mu \boxright \bigl( \mu^{\boxplus t} \bigr) 
& = \mu \boxright \Bigl( \, \bigl( \mu^{\boxplus t} \bigr) 
\boxplus \delta \, \Bigr) &  \mbox{ ($\ncdirac$ neutral 
                  element for $\boxplus$) }                  \\
& = \bB_t \bigl( \mu \boxright  \delta \, \bigr) \ 
       &    \mbox{ (by Proposition \ref{prop:5.1}) }          \\
& = \bB_t ( \mu )   &   \mbox{ (by (\ref{eqn:5.22})). } 
\end{align*}
Then recall from Remark \ref{rem:1.2}.1 that 
$\bigl( \mu^{\boxplus s} \bigr) \boxright
\bigl( \mu^{\boxplus t} \bigr) = 
\Bigl( \mu  \boxright \bigl( \mu^{\boxplus t} \bigr) 
\Bigr)^{\boxplus s}$, and (\ref{eqn:5.31}) follows.
\end{proof}

\begin{remark}    \label{rem:5.35}
The remaining part of this section discusses the relation to 
free Brownian motion stated in Theorem \ref{thm:1.8} of the 
introduction. Same as in Remark \ref{rem:1.7}, we denote by 
$\gamma \in \cD_c (k)$ the joint distribution of a free family
of $k$ centered semicircular elements of variance 1. A fundamental 
property of $\gamma$ is that its $R$-transform is
\begin{equation}     \label{eqn:5.351}
R_{\gamma} (z_1, \ldots , z_k) = z_1^2 + \cdots + z_k^2
\end{equation}
(see e.g. \cite{NS06}, Example 11.21.2 on page 187). More 
generally, for every $t>0$ let $\gamma_t$ denote the 
distribution of a free family of $k$ centered semicircular 
elements of variance $t$. It is immediate that 
\[
R_{\gamma_t} (z_1, \ldots , z_k) = 
R_{\gamma} ( \sqrt{t} z_1, \ldots , \sqrt{t} z_k )
= t( z_1^2 + \cdots + z_k^2);
\]
hence $R_{\gamma_t} = t R_{\gamma}$, which shows that 
$\gamma_t = \gamma^{\boxplus t}$ for every $t > 0$.
\end{remark}

\begin{proposition}   \label{prop:5.4}
Let $\nu$ be a distribution in $\dalg (k)$. One has that
\begin{equation}   \label{eqn:5.71}
R_{\gamma \boxright \nu} (z_1, \ldots , z_k) = \sum_{i=1}^k 
z_i \bigl( 1 + M_{\nu} (z_1, \ldots , z_k) \bigr) z_i .
\end{equation}
\end{proposition}

\begin{proof}
For $n \geq 3$ and $1 \leq i_1, \ldots , i_n \leq k$ one has
\[
\cf_{ (i_1, \ldots , i_n) } \bigl( R_{\gamma \boxright \nu} \bigr)
= \sum_{ \begin{array}{c}
{\scriptstyle \pi \in NC(n),}  \\
{\scriptstyle \pi \leqleq 1_n} 
\end{array} } \ \cf_{ (i_1, \ldots , i_n); \pi } 
( R_{\gamma}, R_{\nu} ) \ 
\ \mbox{ (by Proposition \ref{prop:3.3}) }                  
\]
\[
= \sum_{ \begin{array}{c}
{\scriptstyle \pi \in NC(n) \ such}  \\
{\scriptstyle that \ \{ 1,n \} \in \pi} 
\end{array} } \ \ \delta_{i_1, i_n} \cdot 
\prod_{  \begin{array}{c}
{\scriptstyle W \in \pi}  \\
{\scriptstyle W \neq \{ 1,n \} }
\end{array} } \ \cf_{ (i_1, \ldots , i_n) | W } ( R_{\nu} ) \ 
\ \mbox{ (because of the special form of $R_{\gamma}$). }          
\]
But the set of partitions $\pi \in NC(n)$ which have $\{ 1,n \}$ 
as a block is in natural bijection with $NC(n-2)$; when we follow
through with this bijection, the above sequence of equalities 
is continued with
\begin{align*}
{  }
& = \delta_{i_1, i_n} \cdot \sum_{\rho \in NC(n-2)} \
\prod_{W \in \rho} \
\cf_{ (i_2, \ldots , i_{n-1}) | W } ( R_{\nu} )                  \\
& = \delta_{i_1, i_n} \cdot \nu (X_{i_2} \cdots X_{i_{n-1}} )
\ \ \mbox{ (by the moment-cumulant formula (\ref{eqn:2.141})) }  \\
& = \cf_{ (i_1, \ldots , i_n) } \Bigl( \, \sum_{i=1}^k
z_i \bigl( 1 + M_{\nu} (z_1, \ldots , z_k) \bigr) z_i \, \Bigr) .
\end{align*}

The above calculation shows that the series on the two sides of 
Equation (\ref{eqn:5.71}) have identical coefficients of length 
$\geq 3$. It is immediately verified that the coefficients of 
length 1 and 2 also coincide (each of the two series has 
vanishing linear part and quadratic part equal to 
$\sum_{i=1}^k z_i^2$), and this completes the proof.
\end{proof}

\begin{corollary}   \label{cor:5.5}
The transformation $\Phi : \dalg (k) \to \dalg (k)$ from 
\cite{BN09} satisfies
\begin{equation}    \label{eqn:5.51}
\gamma \boxright \nu = \bB \bigl( \, \Phi ( \nu ) \, \bigr),
\ \ \forall \, \nu \in \dalg (k).
\end{equation}
\end{corollary}

\begin{proof}
In \cite{BN09} the distribution $\Phi ( \nu )$ is defined via
the prescription that its $\eta$-series is
\begin{equation}    \label{eqn:5.81}
\eta_{\Phi ( \nu )} (z_1, \ldots , z_k) = \sum_{i=1}^k
z_i \bigl( 1 + M_{\nu} (z_1, \ldots , z_k) \bigr) z_i.
\end{equation}
Comparing this to Proposition \ref{prop:5.4} we see that
$\eta_{\Phi ( \nu )}$ coincides with the $R$-transform of 
$\gamma \boxright \nu$, and Equation (\ref{eqn:5.51}) follows.
\end{proof}

$\ $

It is worth noting that the two main 
facts proved about $\Phi$ in \cite{BN09} can be easily obtained
from the prespective of subordination distributions, as 
explained in the next proposition. (The two statements of this
proposition originally appeared as Theorem 6.2 and respectively 
as Corollary 7.10 in \cite{BN09}.)

\begin{proposition}    \label{prop:5.6}
$1^o$ For every $\nu \in \dalg (k)$ and $t > 0$ one has that 
\begin{equation}   \label{eqn:5.91}
\Phi ( \nu \boxplus \gamma_t ) = \bB_t ( \, \Phi ( \nu ) \, ).
\end{equation}

$2^o$ The transformation $\Phi$ maps the subset $\cD_c (k)$ of
$\dalg (k)$ into itself.
\end{proposition}

\begin{proof} $1^o$ Since the Boolean Bercovici-Pata bijection 
is one-to-one on $\dalg (k)$, it will suffice to prove that 
\[
\bB \bigl( \, \Phi ( \nu \boxplus \gamma_t ) \, \bigr) = 
\bB \bigl( \, \bB_t ( \, \Phi ( \nu ) \, ) \, \bigr) .
\]
And indeed, starting from the right-hand side of the above 
equation we can go as follows:
\begin{align*}
\bB \bigl( \, \bB_t ( \Phi ( \nu ) ) \, \bigr) 
& = \bB_t \bigl( \, \bB ( \Phi ( \nu ) ) \, \bigr) &   \mbox{ 
 (because $\bB \circ \bB_t = \bB_{t+1} = \bB_t \circ \bB$) }  \\
& = \bB_t \bigl( \, \gamma \boxright \nu \, \bigr) \ 
      &   \mbox{ (by Corollary \ref{cor:5.5}) }               \\
& = \gamma \boxright \bigl( \, \gamma^{\boxplus t} 
    \boxplus \nu \, \bigr) 
     &   \mbox{ (by Proposition \ref{prop:5.1}) }             \\
& = \gamma \boxright ( \nu \boxplus \gamma_t ) 
      &  \mbox{ (because $\gamma^{\boxplus t} = \gamma_t$) }  \\
& = \bB \bigl( \, \Phi ( \nu \boxplus \gamma_t ) \, \bigr) 
      &  \mbox{ (by Corollary \ref{cor:5.5}). } 
\end{align*}

$2^o$ Since $\bB$ is one-to-one, it will suffice to show that 
for $\nu \in \cD_c (k)$ one has $\bB ( \, \Phi ( \nu ) \, ) \in$
$\bB \bigl( \, \cD_c (k) \, \bigr)$. The latter 
set is precisely the set of distributions in $\cD_c (k)$ which 
are $\boxplus$-infinitely divisible (cf. Theorem 1 in 
\cite{BN08}). In view of (\ref{eqn:5.51}), what we have thus to 
prove is the implication
``$\nu \in \cD_c (k) \ \Rightarrow \ \gamma \boxright \nu$
is infinitely divisible''. But $\gamma$ is itself infinitely 
divisible (since $\gamma^{\boxplus t} = \gamma_t \in \cD_c (k)$,
$\forall \, t > 0$), so the required implication follows
from Corollary \ref{cor:4.11}.1.
\end{proof}

$\ $

$\ $

\begin{center}
{\bf\large 6. Properties originating from functional equations}
\end{center}
\setcounter{section}{6}
\setcounter{equation}{0}
\setcounter{theorem}{0}

\begin{remark}   \label{rem:6.1} 
In this remark we briefly return to the 1-variable framework and 
notations from Section 2A, and review the two functional equations
that are to be extended to multi-variable framework.
Recall in particular that for a probability measure $\mu$ on $\bR$,
$F_{\mu} : \bC^{+} \to \bC^{+}$ denotes the reciprocal Cauchy 
transform of $\mu$. In the case when $\mu$ is compactly supported
$F_{\mu} (z)$ can be viewed as a Laurent series in $z$, related to 
the $\eta$-series of $\mu$ by the formula
\begin{equation}     \label{eqn:6.05}
F_{\mu} (z) = z \Bigl( \, 1 - \eta_{\mu} \bigl( \frac{1}{z} \bigr) 
                       \, \Bigr).
\end{equation}
In order to verify (\ref{eqn:6.05}), one writes 
$F_{\mu} = 1/G_{\mu}$, $\eta_{\mu} = M_{\mu}/(1+ M_{\mu})$, and uses 
the relation between $M_{\mu}$ and $G_{\mu}$ that was recorded in 
Equation (\ref{eqn:2.22}) in Section 2A.

\vspace{6pt}

$1^o$ Let $\mu, \nu$ be two probability measures on $\bR$, and let
$\omega_1 , \omega_2$ be the subordination functions of 
$\mu \boxplus \nu$ with respect to $\mu$ and to $\nu$, respectively.
A remarkable equation satisfied by these functions (see e.g.
Theorem 4.1 in \cite{BB07}) is that 
\begin{equation}    \label{eqn:6.11}
\omega_1 (z) + \omega_2 (z) = z + F_{\mu \boxplus \nu} (z), 
\ \ z \in \bC^{+}.
\end{equation}
But $\omega_1 = F_{\nu \boxright \mu}$ and 
$\omega_2 = F_{\mu \boxright \nu}$, hence (\ref{eqn:6.11}) amounts to
\begin{equation}    \label{eqn:6.12}
F_{\mu \boxright \nu} (z) + F_{\nu \boxright \mu} (z)
= z + F_{\mu \boxplus \nu} (z), \ \ z \in \bC^{+}.
\end{equation}
Let us moreover replace the reciprocal Cauchy transforms in 
(\ref{eqn:6.12}) by $\eta$-series, by using Equation 
(\ref{eqn:6.05}). Then (\ref{eqn:6.12}) becomes
\[
\eta_{\mu \boxright \nu} + \eta_{\nu \boxright \mu} 
= \eta_{\mu \boxplus \nu} , 
\]
and in this form it goes through to the multi-variable framework 
of $\dalg (k)$, as shown in Proposition \ref{prop:6.2} below.

\vspace{6pt}

$2^o$ Let $\nu$ be a probability measure on $\bR$. Then for 
every $p \geq 1$ one can consider the probability measure 
$\nu^{\boxplus p}$, and in Theorem 2.5 of \cite{BB04} it was shown 
that one has 
\begin{equation}    \label{eqn:6.13}
G_{\nu^{\boxplus p}} (z) = G_{\nu} \Bigl( \,
\frac{1}{p} z + \bigl( 1 - \frac{1}{p} \bigr) 
F_{\nu^{\boxplus p}} (z) \, \Bigr), \ \ z \in \bC^{+}.
\end{equation}
In other words, Equation (\ref{eqn:6.13}) says that the Cauchy
transform of $\nu^{\boxplus p}$ is subordinated to the one of 
$\nu$, with subordination function $\omega$ defined by 
\begin{equation}   \label{eqn:6.14}
\omega (z) = \frac{1}{p} z + \bigl( 1 - \frac{1}{p} \bigr)
F_{\nu^{\boxplus p}} (z) , \ \ z \in \bC^{+} .
\end{equation}
It is immediate that $\omega$ from (\ref{eqn:6.14}) belongs to 
the set $\fF$ of reciprocal Cauchy transforms from Equation
(\ref{eqn:2.13}) of Section 2A, hence there exists a unique 
probability measure $\sigma$ on $\bR$ such that $F_{\sigma} = \omega$.
It is natural to call this $\sigma$ the ``subordination distribution
of $\nu^{\boxplus p}$ with respect to $\nu$''. (If $p \geq 2$ then
$\sigma$ is just $\nu^{\boxplus (p-1)} \boxright \nu$, but for 
$1 \leq p < 2$ this point of view doesn't always work, as the
probability measure $\nu^{\boxplus (p-1)}$ might not be defined.)
So then Equation (\ref{eqn:6.14}) becomes
\[
F_{\sigma} (z) = \frac{1}{p} z + \Bigl( 1 - \frac{1}{p} \bigr)
F_{\nu^{\boxplus p}} (z) , \ \ z \in \bC^{+} ,
\]
and upon writing the reciprocal Cauchy transforms in terms of 
$\eta$-series this takes us to 
\begin{equation}    \label{eqn:6.15} 
\eta_{\sigma} = \frac{p-1}{p} \cdot \eta_{ \nu^{\boxplus p} } .
\end{equation}
This latter formula is the one that will be extended to the 
framework of $\cD_c (k)$ -- see Corollary \ref{cor:6.4} and
Remark \ref{rem:6.5} below.
\end{remark}

\begin{proposition}    \label{prop:6.2}
For every $\mu, \nu \in \dalg (k)$ one has that 
\begin{equation}  \label{eqn:3.121}
\eta_{\mu \boxplus \nu} =  \eta_{\mu \boxright \nu}  +
\eta_{\nu \boxright \mu} .
\end{equation}
\end{proposition}

\begin{proof} 
We fix $n \in \bN$ and $1 \leq i_1, \ldots , i_n \leq k$ and compare
the coefficients of $z_{i_1} \cdots z_{i_n}$ for the series on the 
two sides of Equation (\ref{eqn:3.121}). By using the relation 
between $R$ and $\eta$ and the linearizing property of $R$, and 
then by invoking Equation (\ref{eqn:2.121}) in Remark \ref{rem:2.12}
we find that
\begin{align*} 
\cf_{ (i_1, \ldots , i_n) } \bigl( \eta_{\mu \boxplus \nu} \bigr)
& = \sum_{ \begin{array}{c} 
       {\scriptstyle \pi \in NC(n),}  \\
       {\scriptstyle \pi \leqleq 1_n} \end{array}  } 
       \ \cf_{ (i_1, \ldots , i_n); \pi } (R_{\mu} + R_{\nu})  \\ 
& = \sum_{ \begin{array}{c} 
       {\scriptstyle \pi \in NC(n),}  \\
       {\scriptstyle \pi \leqleq 1_n} \end{array}  } \
       \ \sum_{ c: \pi \to \{ 1,2 \} } \
       \ \cf_{ (i_1, \ldots , i_n); \pi ;c } (R_{\mu}, R_{\nu}).
\end{align*}  
In the latter double sum, the colourings $c$ of $\pi$ can be 
subdivided according to whether $c(V_0) =1$ or $c(V_0) =2$, where 
$V_0$ is the unique outer block of $\pi$. This leads to an 
equality of the form
\[
\cf_{ (i_1, \ldots , i_n) } \bigl( \eta_{\mu \boxplus \nu} \bigr)
= \Sigma_1 + \Sigma_2,
\]
where $\Sigma_1$ is exactly as on the right-hand side of Equation 
(\ref{eqn:3.71}) from Proposition \ref{prop:3.7}, and $\Sigma_2$
is the counterpart of $\Sigma_1$ with the roles of $\mu$ and $\nu$
being reversed. We are only left to invoke Proposition 
\ref{prop:3.7} to conclude that 
\[
\Sigma_1 + \Sigma_2
=  \cf_{ (i_1, \ldots , i_n) } ( \eta_{\mu \boxright \nu} )
    + \cf_{ (i_1, \ldots , i_n) } ( \eta_{\nu \boxright \mu} )  
=  \cf_{ (i_1, \ldots , i_n) } \bigl( \eta_{\mu \boxright \nu} 
+ \eta_{\nu \boxright \mu} \bigr) ,
\]
and (\ref{eqn:3.121}) follows.
\end{proof}

$\ $

When discussing the multi-variable analogue for Equation
(\ref{eqn:6.15}) it is convenient to note that there 
is no problem to generally talk about the 
``subordination distribution of $\lambda$ with respect to $\nu$''
for any $\lambda, \nu \in \dalg (k)$.

\begin{definition}    \label{def:6.3}
Let two distributions $\lambda , \nu \in \dalg (k)$ be given. 
Consider the distribution $\mu \in \dalg (k)$ which is uniquely 
determined by the requirement that 
\begin{equation}   \label{eqn:6.31}
R_{\mu} = R_{\lambda} - R_{\nu}
\end{equation}
(or equivalently, via the requirement that 
$\mu \boxplus \nu = \lambda$). Then the {\em subordination
distribution of $\lambda$ with respect to $\nu$} is, by definition,
equal to $\mu \boxright \nu$.
\end{definition}

\begin{corollary}   \label{cor:6.4}

$1^o$ For every $\nu \in \dalg (k)$ and every $p \geq 1$, the 
subordination distribution of $\nu^{\boxplus p}$ with respect 
to $\nu$ is equal to $\bigl( \bB ( \nu ) \bigr)^{\boxplus (p-1)}$.

$2^o$ Let $\nu$ be a distribution in $\cD_c (k)$. Then, for every 
$p \geq 1$, the subordination distribution of $\nu^{\boxplus p}$ 
with respect to $\nu$ belongs to $\cD_c (k)$ as well, and is 
moreover $\boxplus$-infinitely divisible. 
\end{corollary}

\begin{proof} $1^o$ According to Definition \ref{def:6.3}, the 
distribution in question is $\nu^{\boxplus (p-1)} \boxright \nu$. 
Thus we only need to invoke the particular case of Proposition 
\ref{prop:5.3} where $s = p-1$ and $t=1$.

$2^o$ This follows from part $1^o$ of the corollary and the fact
that $\bB ( \mu )$ is $\boxplus$-infinitely divisible (which 
implies that any convolution power 
$\bigl( \bB ( \nu ) \bigr)^{\boxplus t}$, $t \geq 0$,
lives in $\cD_c (k)$ and is itself infinitely divisible).
\end{proof}

\begin{remark}    \label{rem:6.5}
It is an easy exercise (left to the reader) to verify 
the identity 
\begin{equation}   \label{eqn:6.51}
\bigl( \bB ( \nu ) \bigr)^{\boxplus (p-1)} 
= \bigl( \nu^{\boxplus p} \bigr)^{\uplus (p-1)/p} , \ \ 
\forall \, \nu \in \dalg (k), \ \forall \, p \in [ 1, \infty ).
\end{equation}
So if we denote the subordination distribution of $\nu^{\boxplus p}$
with respect to $\nu$ by $\sigma$, then by invoking Corollary 
\ref{cor:6.4} and by taking the $\eta$-series of the distribution
on the right-hand side of (\ref{eqn:6.51}) we obtain that 
$\eta_{\sigma}$ = 
$\bigl( (p-1)/p \bigr) \cdot \eta_{ \nu^{\boxplus p} }$.
Thus Corollary \ref{cor:6.4} gives indeed a multi-variable 
generalization of Equation (\ref{eqn:6.15}) from Remark 
\ref{rem:6.1}.2.
\end{remark}

$\ $

$\ $

$\ $

Alexandru Nica

Department of Pure Mathematics, University of Waterloo,

Waterloo, Ontario N2L 3G1, Canada.

Email: anica@math.uwaterloo.ca

\end{document}